\documentclass[11pt, a4paper, leqno]{amsart}

\usepackage{amsxtra}
\usepackage{graphicx}
\usepackage{newlfont}
\usepackage{amscd}
\usepackage{hyperref}
\usepackage{cancel}
\usepackage{amsmath,amsthm}
\usepackage{amssymb}
\usepackage{enumerate}
\usepackage{color}
\usepackage[all,cmtip]{xy}

 \newtheorem{thm}{Theorem}[section]

 \newtheorem{cor}[thm]{Corollary}
 \newtheorem{lem}[thm]{Lemma}
 \newtheorem{prop}[thm]{Proposition}
 
  \theoremstyle{definition}
 \newtheorem{defn}[thm]{Definition}
  
  \newtheorem{defn-thm}[thm]{Definition-Theorem}
   \newtheorem{ex}[thm]{Example}

 \theoremstyle{remark}

 \newtheorem{rem}[thm]{Remark}

\numberwithin{equation}{section}

\numberwithin{thm}{section}

\numberwithin{table}{section}

\numberwithin{figure}{section}

\ifx\pdfoutput\undefined
  \DeclareGraphicsExtensions{.pstex, .eps}
\else
  \ifx\pdfoutput\relax
    \DeclareGraphicsExtensions{.pstex, .eps}
  \else
    \ifnum\pdfoutput>0
      \DeclareGraphicsExtensions{.pdf}
    \else
      \DeclareGraphicsExtensions{.pstex, .eps}
    \fi
  \fi
\fi

\newcommand{\ZZ}{\mathbb{Z}}
\newcommand{\NN}{\mathbb{N}}

\newcommand{\CC}{\mathbb{C}}

\newcommand{\g}{\mathfrak{g}}

\newcommand{\m}{\mathfrak{m}}
\newcommand{\n}{\mathfrak{n}}

\newcommand{\p}{\mathfrak{p}}

\newcommand{\ad}{\text{ad}}
\newcommand{\gr}{\text{gr}}
\newcommand{\sll}{\text{sl}}

\newcommand{\WW}{\mathcal{W}}

\begin{document}

\title[]{Structures of  (supersymmetric) classical  W-algebras}
\author{Uhi Rinn Suh}
\address[U.R. Suh]
{ Department of Mathematical Sciences and Research institute of Mathematics, Seoul National University, GwanAkRo 1, Gwanak-Gu, Seoul 08826,
Korea}
\email{uhrisu1@snu.ac.kr}

\begin{abstract}
In the first part of this paper, we discuss the classical W-algebra $\mathcal{W}(\mathfrak{g}, F)$ associated with a Lie superalgebra $\mathfrak{g}$ and the nilpotent element $F$ in an $\mathfrak{sl}_2$-triple.
We find  a generating set  of $\mathcal{W}(\mathfrak{g}, F)$ and compute the  Poisson brackets between them.  In the second part, which is the main part of the paper,  we discuss  supersymmetric classical W-algebras.
 We introduce two different constructions of a supersymmetric classical W-algebra $\mathcal{W}(\mathfrak{g}, f)$ associated with a Lie superalgebra $\mathfrak{g}$ and an odd nilpotent element $f$ in a subalgebra isomorphic to $\mathfrak{osp}(1|2)$. The first construction is via the SUSY classical BRST complex and the second is via the SUSY Drinfeld-Sokolov Hamiltonian reduction.
We show that these two methods give rise to isomorphic SUSY Poisson vertex algebras. As a supersymmetric analogue of the first part, we compute explicit generators and Poisson brackets between the generators.

\end{abstract}

\maketitle

\section{Introduction}

W-algebras were  introduced by Zamolodchikov \cite{Zam85} and Fateev-Lukyanov \cite{FL88} in the conformal field theory (see the review article by Bouwknegt-Schoutens \cite{BS93} and the references therein). Drinfeld-Sokolov \cite{DS} described the classical W-algebra $\WW(\g_{\text{ev}})$ associated with a Lie algebra $\g_{\text{ev}}$ via  a Hamiltonian reduction and related it to a hierarchy of integrable  Hamiltonian systems.  Feigin-Frenkel \cite{FF90} constructed the  (quantum) W-algebra $W(\g_{\text{ev}})$  via the BRST quantization of the Drinfeld-Sokolov reduction. As a generalization,  Kac-Roan-Wakimoto \cite{KRW}  introduced W-algebras $W(\g, F)$ associated with a finite simple Lie superalgebra $\g$ and a nilpotent element $F$ in an $\mathfrak{sl}_2$-triple (see also \cite{DK, KW}). Note that the  classical and quantum W-algebras $\WW(\g_{\text{ev}})$ and $W(\g_{\text{ev}})$, respectively, by Drinfeld-Sokolov and Feigin-Frenkel are the cases where  $F$  is a principle nilpotent element. 

 A classical W-algebra $\WW(\g,F)$ can be understood as the classical limit of the corresponding quantum W-algebra $W(\g, F)$. More precisely, $\WW(\g,F)$ can be equivalently defined  
 via  the classical BRST complex  and the generalized Drinfeld-Sokolov reduction \cite{DKV13, S15}.
 The second construction shows the similarities between classical  W-algebras and finite (classical) W-algebras \cite{GG, Pre}. 

\vskip 2mm

Classical W-algebras are Poisson vertex algebras (PVAs) and PVAs are the underlying algebraic structures of infinite dimensional Hamiltonian systems. Barakat-De Sole-Kac \cite{BDK} investigated integrable Hamiltonian systems in terms of PVAs.  In a series of papers by De Sole-Kac-Valeri \cite{DKV18, DKV15, DKV13, DKV14, DKV16_2}, Hamiltonian integrable systems related to the classical W-algebras $\WW(\g_{\text{ev}}, F)$  were constructed. Moreover, they  computed explicit forms of generators of $\WW(\g_{\text{ev}},F)$ and Poisson brackets between the the generators  \cite{DKV16_1}.

The author  investigated in \cite{S16} the algebraic properties of classical W-algebras $\WW(\g,F)$ associated with Lie superalgebras $\g$. The following properties are crucially used in this paper:
\begin{itemize}
\item As a differential algebra, $\WW(\g,F)\simeq S(\CC[\partial]\otimes \text{ker}\, F),$ where $\partial$ is the even derivation in the supersymmetric algebra $S(\CC[\partial]\otimes \text{ker}\, F)$.
\item  $\WW(\g, F)=\CC[\partial^n J_i|n\in \ZZ_{\geq 0}, i\in I]$ for $J_i= u_i+ A_i$  where $\{u_i|i\in I\}$ is a basis of $ \text{ker} \, \text{ad} F$ and $A_i$ consists of total derivative parts and degree $\geq 2$ parts. 
\end{itemize} 
Using the above mentioned properties, explicit forms of generators and Poisson brackets between them are presented in Section \ref{Sec: generators of W-alg}. In addition, a partial result on related Hamiltonian integrable systems has been written \cite{S18}.

\vskip 2mm

The supersymmetric counterpart of the W-algebras, the so-called supersymmetric (SUSY) W-algebras were introduced by Madsen-Ragoucy \cite{MR} via SUSY BRST complexes. A SUSY W-algebra $W(\g, f)$ is constructed for  simple Lie superalgebra $\g$ and odd nilpotent $f$ in a subalgebra isomorphic to $\mathfrak{osp}(1|2)$. In the article by Heluani-Kac \cite{HK06}, they provided the notion and basic properties of SUSY vertex algebras (see also \cite{K98}). Molev, Ragoucy and the author \cite{MRS19} showed that the SUSY BRST complexes and SUSY W-algebras can be described within  SUSY vertex algebra theory.  

SUSY classical  W-algebras, the most crucial ingredients of this paper,
 have been studied as Poisson structures of 
 super integrable systems \cite{CS19, DG98, IK92, IK91, I91, IK91_2, Kuper, KZ, ManRa85, Mc92, OP90}.  In particular, in various articles such as \cite{CS19, DG98, IK91, IK92}, super integrable systems associated with Lie superalgebras have been constructed and the related Poisson algebras are called  SUSY W-algebras associated with Lie superalgebras. These arguments can be viewed as supersymmetric generalizations of Drinfeld-Sokolov reductions and hierarchies. 

\vskip 2mm

The first part of this paper  (Sections \ref{Sec:PVA and W-alg} and \ref{Sec: generators of W-alg}) deals with (nonSUSY) classical W-algebras $\WW(\g,F)$. In Section \ref{Sec:PVA and W-alg}, we review the basic properties of PVAs and classical W-algebras. In particular, the crucial facts used in Section \ref{Sec: generators of W-alg} are the master formula for Poisson $\lambda$-brackets (see Proposition \ref{Prop:master} ) and the properties of free generators of 
$\WW(\g,F)$ (see Proposition \ref{Prop:structure of W}). Using these propositions, we find 
 generators of $\WW(\g,F)$ and  Poisson $\lambda$-brackets between the generators. These results are analogous to the results for  classical W-algebras associated with Lie algebras in \cite{DKV16_1}.

The second part (Section \ref{Sec:SUSY PVA}, \ref{Sec:SUSY_W}, \ref{generators of SUSY W}) is the supersymmetric analogue of the first part. Accordingly, we review the basics of SUSY Poisson vertex algebras in Section \ref{Sec:SUSY PVA}. In Section \ref{Sec:SUSY_W}, we clarify the definition of a SUSY classical W-algebra $\WW(\g, f)$. There are two natural ways to construct $\WW(\g, f)$:
\begin{enumerate}[(i)]
\item via SUSY classical BRST complexes (Theorem \ref{Thm: SISU W_ BRST})
\item using SUSY Drinfeld-Sokolov Hamiltonian reductions (Definition \ref{Def: first def/SUSY W})
\end{enumerate} 
In Theorem \ref{Thm:first_third_equiv}, we prove that (i) and (ii) are isomorphic as SUSY PVAs. The SUSY classical BRST complexes can be obtained by the classical limit of SUSY BRST complexes \cite{MRS19}.  The classical limit of SUSY vertex algebras has been studied by Heluani-Kac \cite{HK06}. As for quantum and classical W-algebras and SUSY quantum W-algebras \cite{KRW, KW, DK, S16, MRS19}, BRST complexes allow us to find the number of free generators of $\WW(\g, f)$ and the following properties of generators (see Proposition \ref{Prop:5.5_0408}):
\begin{itemize}
\item As a differential algebra, $\WW(\g, f)\simeq S(\CC[D]\otimes \text{ker} \, \text{ad} f)$, where  $D$ is the odd derivation. 
\item $\WW(\g, f)=\CC[D^n J_i| n\in \ZZ_{\geq 0}, i\in I_S]$ for $J_i= u_i+ A_i$,  where $\{u_i|i\in I_S\}$ is a basis of $ \text{ker} \, \text{ad} f$ and $A_i$ consists of total derivative parts and degree $\geq 2$ parts. 
\end{itemize}  
Furthermore, using the second definition (ii) of $\WW(\g, f)$, we find its algebraic structures analogous to those of nonSUSY classical W-algebras. More precisely, we can find the explicit forms of elements in $\WW(\g, f)$ by finding $J_i$'s (Theorem \ref{Thm:SUSY_generator}). For the Poisson $\chi$-brackets between generators (Theorem \ref{Thm:Pisson chi bracket}), we use the SUSY master formula  \cite{CS19}. We plan to study relations between the nonSUSY classical W-algebra $\WW(\g, F=f^2)$ and the SUSY classical W-algebra $\WW(\g, f)$ with Victor Kac, as part of our future work.

\vskip 2mm

\textit{Acknowledgments} The author was supported by the New Faculty Startup Fund from Seoul National University and the NRF Grant \# NRF-2019R1F1A1059363.

\section{Poisson vertex algebras and classical W-algebras} \label{Sec:PVA and W-alg}

In this section, we recall notions related to Poisson vertex algebras introduced in \cite{BK} and classical W-algebras. Note that, throughout this paper, the base field is $\CC$. 

\subsection{Poisson vertex algebras}

A $\ZZ/2\ZZ$-graded vector space $V=V_{\bar{0}}\oplus V_{\bar{1}}$ is called a vector superspace and a homogeneous element $a\in V_{\bar{i}}$ is called an even (resp. odd) element if $i=0$ (resp. $i=1$). Denote $p(a):=i$, $s(a)=(-1)^{i}$ and $s(a,b)= (-1)^{ij}$ for $a\in V_{\bar{i}}$ and $b\in V_{\bar{j}}$. An operator  $\phi: V \to V$ is called even if $\phi(V_{\bar{i}}) \subset V_{\bar{i}}$  and  odd if $\phi(V_{\bar{i}})\subset V_{\bar{i}+\bar{1}}$. A  $\CC$-algebra $A$ is  called a superalgebra if $A$ is a vector superspace such that $A_{\bar{i}} A_{\bar{j}} \subset A_{\bar{i}+\bar{j}}$.

Let $R=R_{\bar{0}} \oplus R_{\bar{1}}$ be a $\CC[\partial]$-module with  even operator $\partial:R\to R$. Consider the  even indeterminate $\lambda$ and  $\CC[\partial]$-module $\CC[\lambda] \otimes R$  via 
\[ \partial (\lambda^n \otimes a) = \lambda^n \otimes  \partial a \]
for $a\in R$. Note that we usually omit $\otimes$ in elements of $\CC[\lambda] \otimes R$.

\begin{defn}\label{Def:LCA} Suppose the $\CC[\partial]$-module $R$ endowed with the linear $\lambda$-bracket 
\[ [ \ {}_\lambda \ ] : R \otimes_\CC R \to \CC[\lambda] \otimes R \]
satisfies the following properties:
\begin{itemize}
\item (parity preserving)  $p(a)+p(b)= p([a{}_\lambda b])$,
\item (sesquilinearity) $[a{}_\lambda \partial b]= (\partial+\lambda)[a_\lambda b]$ and $[\partial a{}_\lambda b]=-\lambda[a{}_\lambda b]$,
\item (skew-symmetry) $[a{}_\lambda b]= -s(a,b) [b_{-\lambda-\partial} a],$
\item (Jacobi identity) $[a_\lambda [b_\mu c]] = [[a_\lambda b]_{\lambda+\mu} c]+s(a,b)[b_\mu [a_\lambda c]]$ in $\CC[\lambda,\mu]\otimes R,$
\end{itemize}
for $a,b,c\in R$. 
Then $R$ is called a {\it Lie conformal algebra (LCA).}
\end{defn}

To be  precise, if $R$ is a LCA and we write 
\begin{equation} \label{lambda_notation}
\textstyle[a{}_\lambda b]=\sum_{n\in \ZZ_{\geq 0}} \lambda^n a_{(n)} b
\end{equation}
for  $a,b,a_{(n)}b\in R$, terms in the skew-symmetry and Jacobi identity are 
\begin{itemize}
\item $[b{}_{-\lambda-\partial} a]=\sum_{n\in \ZZ_{\geq 0}} (-\lambda-\partial)^n a_{(n)} b$,
\item $[a {}_\lambda [b{}_\mu c]]= \sum_{n,m\in \ZZ_{\geq 0}}\lambda^n \mu^m a_{(n)}(b_{(m)}c) $, $[b {}_\mu [a{}_\lambda c]]= \sum_{n,m\in \ZZ_{\geq 0}}\mu^n \lambda^m b_{(n)}(a_{(m)}c)$,
\item $[[a{}_\lambda b]_{\lambda+\mu} c]=\sum_{n,m\in \ZZ_{\geq 0}} \lambda^n (\mu+\lambda)^m (a_{(n)}b)_{(m)}c$,
\end{itemize}
where $\lambda \mu= \mu \lambda$ and $\mu \partial= \partial \mu$. 

\begin{rem}
In other articles such as \cite{DK, DKV16_1,S16, S18}, the coefficient of $\lambda^n$ in $[a {}_\lambda b]$  is denoted by $\frac{1}{n!}a_{(n)}b$. However, for simplicity of notation, we denote the coefficient by $a_{(n)}b$ as in  \eqref{lambda_notation}.
\end{rem}

\begin{defn}\label{Def:PVA}
A quadruple $(P, \partial, \{{}_\lambda {}\}, \cdot)$ is a {\it Poisson vertex algebra (PVA)} if it satisfies 
\begin{itemize}
\item  $(P, \partial, \cdot )$ is a unital super-commutative differential algebra with even derivation $\partial: P \to P$, i.e., $\partial(ab)= \partial(a) b + a \partial(b)$ for $a, b\in P$,
\item  $(P, \partial, \{{}_\lambda{}\})$ is an LCA,
\item (Leibniz rule) $\{a{}_\lambda bc\}=s(a,b) b\{a{}_\lambda c\}+\{a{}_\lambda b\} c$. 
\end{itemize}
\end{defn}

Note that the Leibniz rule in Definition \ref{Def:PVA} is also called the right Leibniz rule. By the skew-symmetry and right Leibniz rule of LCAs, the left Leibniz rule
\begin{equation} \label{Eqn:LLR}
\textstyle  \{ab{}_\lambda c\}=  s(b,c)\{a_{\lambda+\partial}c\}_{\to} b + s(a,bc)\{b{}_{\lambda+\partial}c\}_{\to}a  \end{equation}
holds. Here, 
\[\textstyle \{a_{\lambda+\partial} c\}_{\to} b= \sum_{n\in \ZZ_{\geq 0}} a_{(n)}c \, (\lambda+\partial)^n b.\]

As in Definition \ref{Def:PVA}, a PVA is a LCA. On the other hand, for a given LCA, one can generate a PVA in the following manner. We refer to \cite{BDK, S18} for the following propositions (Proposition \ref{Prop:VS,LCA to PVA} and Proposition \ref{Prop:master}).

\begin{prop} \label{Prop:VS,LCA to PVA} \ 
\begin{enumerate}
\item
Let $R$ be a LCA. Then the supersymmetric algebra $S(R)$ generated by $R$ is a PVA endowed with the $\lambda$-bracket defined by the $\lambda$-bracket on $R$ and the Leibniz rule.
\item 
Let $V$ be a vector superspace and  $P=S(\CC[\partial]\otimes V)$, the supersymmetric algebra generated by $\bigoplus_{n\in \ZZ_{\geq 0}}\partial^n V$. If a bracket $[\ {}_\lambda {}\ ]: V\otimes V \to \CC[\lambda]\otimes P$ satisfies the skew-symmetry and Jacobi identity then it can be uniquely extended to a PVA bracket on $P$ via the sesquilinearity of LCAs and the Leibniz rule of PVAs. 
\end{enumerate}
\end{prop}

Let $V=V_{\bar{0}}\oplus V_{\bar{1}}$ be a vector superspace and let $\mathcal{B}=\mathcal{B}_{\bar{0}} \cup \mathcal{B}_{\bar{1}}$ be a basis of $V$ such that $\mathcal{B}_{\bar{0}}:=\{u_i|i\in I_{\bar{0}}\}$ and  $\mathcal{B}_{\bar{1}}:=\{u_i|i\in I_{\bar{1}}\}$ are bases of $V_{\bar{0}}$ and $V_{\bar{1}}$, respectively. Consider 
 the  supersymmetric differential algebra $P=S(\CC[\partial]\otimes V)$  generated by $V$ and denote $u_i^{(m)}:= \partial^m u_i$ for $m\in \ZZ_{\geq 0}$.  Then, as an algebra of polynomials, $P= \CC[\, u_i^{(m)}\, |\, i\in  I_{\bar{0}} \cup I_{\bar{1}}, m\in \ZZ_{\geq 0}\, ]$.

For $u_i \in \mathcal{B}$, define the derivation $\frac{\partial}{\partial \, u^{(m)}_i}$ on $P$ of parity $p(u_i)$ via
\[ \textstyle \frac{\partial}{\partial u^{(m)}_i} u_j^{(n)}=\delta_{m,n} \delta_{i,j}, \quad  \frac{\partial}{\partial u^{(m)}_i} (fg) = \frac{\partial}{\partial u^{(m)}_i} (f)g +s(u_i, f) \, f\frac{\partial}{\partial u^{(m)}_i} (g).   \]

\begin{prop}[Master formula for PVA] \label{Prop:master}  
If the differential supersymmetric algebra $P$ is a PVA endowed with the bracket $\{{}_\lambda \}$ then 
\[ \{f{}_\lambda g\}= \sum_{\substack{i,j\in I \\ m,n\in \ZZ_{\geq 0}} } C^{f,g}_{i,j} \frac{\partial g}{\partial u_j^{(n)}} (\lambda+\partial)^n \{u_i{}_{\lambda+\partial} u_j \}_{\to}(-\lambda-\partial)^m \frac{\partial f}{\partial u_i^{(m)}},    \]
where $C^{f,g}_{i,j}:= s(f,g)s(u_i,u_j) s(g,u_i)s(u_i)$ for $f,g\in P$.
\end{prop}

\begin{ex}
Let $\g$ be a Lie superalgebra with an invariant bilinear form $(\, |\, )$. The supersymmetric differential algebra  $P=S(\CC[\partial]\otimes \g)$ generated by $\g$ is called the affine PVA if it is endowed with the bracket defined by 
\begin{equation}\label{affine_bracket}
 \{ a{}_\lambda b\}= [a,b]+k\lambda(a|b), \quad a,b\in \g
 \end{equation}
for  $k\in \CC$. By Proposition \ref{Prop:VS,LCA to PVA}, \eqref{affine_bracket} completely determines the PVA structure of $P$ and, by Proposition \ref{Prop:master}, the $\lambda$-bracket of any two elements can be computed directly.
\end{ex}

\subsection{Classical W-algebras}

Let $\g$ be a finite simple Lie superalgebra with an $\sll_2$-triple $(E,H,F)$ and the nondegenerate even supersymmetric invariant bilinear form $(\, |\, )$ such that $(E|F)=\frac{1}{2}(H|H)=1$. Then 
\begin{equation}  \label{Eqn:grading in g}
\textstyle\g=\bigoplus_{i\in \frac{\ZZ}{2}} \g(i)
\end{equation}
is the eigenspace decomposition with respect to $\text{ad}\frac{H}{2}$.  Let us write  
\begin{equation}\label{subspace_notation}
\textstyle \g_{\geq j} =  \bigoplus_{i\geq j} \g(i), \quad \g_{\leq j} =  \bigoplus_{i\leq j} \g(i), \quad \g_{> j} =  \bigoplus_{i>j} \g(i), \quad \g_{< j} =  \bigoplus_{i< j} \g(i).
\end{equation}
The super Lie subalgebras $\n$ and $\m$ and the subspace $\p$ of $\g$  are defined by 
\begin{equation}
\textstyle \n=\g_{>0}, \quad \m= \g_{\geq 1}, \quad \p=\g_{<1}.
\end{equation}

Consider the differential superalgebra $P(\g)= S(\CC[\partial]\otimes \g)$ and its ideal $\mathcal{I}_F$ generated by $m-(F|m)$ for $m\in \m$. Recall that $P(\g)$ endowed with the $\lambda$-bracket \eqref{affine_bracket} is called the affine PVA. Let $\text{ad}_\lambda \, n : P(\g) \to \CC[\lambda] \otimes P(\g)$ for $n\in \n$ be defined by:
\begin{equation} \label{ad n action}
 \text{ad}_\lambda n(A):= \{n_\lambda A\}  \ \text{ for }\  n\in \n, \,  A\in P(\g).
\end{equation}

Since $\text{ad}_\lambda n (\mathcal{I}_F) \subset \CC[\lambda]  \otimes \mathcal{I}_F$, \eqref{ad n action} canonically induces the function from  $P(\g)/\mathcal{I}_F$ to $\CC[\lambda] \otimes (P(\g)/\mathcal{I}_F)$. Hence, for $A\in P(\g)$, if  we denote the image of $A$ in $P(\g)/\mathcal{I}_F$ by $[A]$ then  
\begin{equation}
\WW(\g,F)= (P(\g)/\mathcal{I}_F)^{\text{ad}_\lambda \n}:= \{\, [A]\in P(\g)/\mathcal{I}_F\, |\,  \{n {}_\lambda A  \} \subset \CC[\lambda] \otimes\mathcal{I}_F\}
\end{equation}
is a well-defined vector superspace. 

\begin{rem} Since $P(\g)/\mathcal{I}_F \simeq S(\CC[\partial]\otimes \p)$ as superalgebras, each element in $P(\g)/\mathcal{I}_F$ is uniquely represented by an element in $S(\CC[\partial]\otimes \p)$. Hence, 
if there is no danger of confusion, we consider $\WW(\g,F)$  a subspace of $S(\CC[\partial]\otimes \p)$.  
\end{rem}

Moreover, it is known that the vector superspace $\WW(\g,F)$ has a PVA structure.

\begin{prop}\cite{DKV13, S16}
$\WW(\g,F)$ is a PVA endowed with the $\lambda$-bracket induced by the $\lambda$-bracket of the affine PVA $P(\g)$.
\end{prop}

\begin{defn}
The PVA $\WW(\g,F)$ is called the classical W-algebra associated with $\g$ and $F$.
\end{defn}

Now, let us define the conformal weight  of $a\in \g(i)$ by
\begin{equation} \label{conformal weight (1)}
\Delta_a=1-i.
\end{equation}
 Then $P(\g)$ has  conformal weight decomposition via  \eqref{conformal weight (1)} and 
\begin{equation} \label{conformal weight (2)}
\Delta_{AB}= \Delta_A +\Delta_B , \quad \Delta_{\partial A}= A+1,
\end{equation}
where  $A,B\in P(\g)$ are homogeneous elements with respect to the conformal weight.

Since the element $E\in \g(1)$ in the $\mathfrak{sl}_2$-triple has a conformal weight of $0$,  the ideal $\mathcal{I}_F$ of $P(\g)$ is homogeneous. Hence we can consider the conformal weight of the PVA  $\WW(\g, F)$ induced from the conformal weight \eqref{conformal weight (1)} of $P(\g)$.

We close this section by presenting  crucial properties of a generating set of $\WW(\g,F)$.

\begin{prop} \cite{DKV13, S16}\label{Prop:structure of W}
Let $\g^F:= \text{ker}\, \text{ad}F$ and $\mathcal{B}^F=\{ q_i| i\in J^F \}$ be a homogeneous basis of $\g^F$. Then there exists a set $\{ w_i|i \in J^F\}\subset \WW(\g, F)$ such that 
\begin{enumerate}[(i)]
\item $\WW(\g,F)=\CC[\partial^n w_i|n\in \ZZ_{\geq 0}, i\in J^F]$ as differential  algebras,
\item $\Delta_{w_i}= \Delta_{q_i}$, 
\item $w_i = q_i  + \text{(total derivative part)}+ \text{(degree $\geq 2$ part as the  polynomial degree)}.$ 
\end{enumerate}
\end{prop}

\begin{rem}\label{Remark:structure of W}
In Proposition \ref{Prop:structure of W}, one can assume that $w_i-q_i$ does not have any monomial in $S(\CC[\partial]\otimes \g^F)$. If $w_i-q_i$ has a monomial $\partial^{n_1} q_{i_1} \partial^{n_2} q_{i_2} \cdots \partial^{n_k} q_{i_k}\in P(\g^F)$ for $n_1, \cdots, n_k \in \ZZ_{\geq 0}$ and $i_1, \cdots, i_k\in J^F$ then $w_i$ can be replaced with  $w_i-\partial^{n_1} w_{i_1} \partial^{n_2} w_{i_2} \cdots \partial^{n_k} w_{i_k}$. 
\end{rem}

\section{Structures of classical W-algebras associated with Lie superalgebras} \label{Sec: generators of W-alg}

In this section, we find free generators of a classical affine W-algebra  $\WW(\g,F)$ and $\lambda$-brackets between them. We refer to   \cite{DKV16_1}  for the analogous results when $\g$ is a Lie algebra.

 As in the previous section, let $\g$ be a simple Lie superalgebra with an $\mathfrak{sl}_2$-triple $(E,H,F)$ and even nondegenerate supersymmetric invariant bilinear form $(\ | \ )$ such that $\frac{1}{2}(H|H)= (E|F)=1$. Then $\g = \bigoplus_{i\in \ZZ/2} \g(i)$ for the eigenspace $\g(i)$ with respect to $\frac{H}{2}$.  

According to the $\mathfrak{sl}_2$ representation theory, there are bases
\begin{equation}
 \{ q_j|j\in J^F \}, \quad \{ q^j |j\in J^F\}
\end{equation}
 of  $\g^F:= \ker \ad F \subset \bigoplus_{i\leq 0} \g(i)$ and $\g^E:= \ker \ad E \subset \bigoplus_{i\geq 0} \g(i)$ such that $(q^i|q_j)= \delta_{i,j}$.  We  assume the bases of $\g^F$ and $\g^E$ are homogeneous with respect to both parity and  the $\frac{\ZZ}{2}$-grading.
For simplicity of notation, we denote  $s(i):= s(q_i)=s(q^i)$ for $i\in J^F$.
 
Again, by the representation of $\mathfrak{sl}_2$,  we have  
\begin{equation} \label{decomp_g}
 \g= \g^F \oplus [E,\g]
\end{equation}
and there is a homogeneous basis  
\begin{equation}\label{basis1}
\mathcal{B}=\{ q^j_n := (\ad F)^n q^j |  j\in J^F, \ n=0,1,\cdots, 2\alpha_j\}
\end{equation}
of $\g$, where $q^j\in \g(\alpha_j)$. 

\begin{lem}\cite{DKV16_1}
Let $\mathcal{B}^*= \{ \, q_j^n \,  |\, j\in J^F, n=1, \cdots, 2g_j\}$ be the dual basis of $\mathcal{B}$ such that $ (q_m^i |q_j^n)= \delta_{i,j} \delta_{m,n}$. Then we have 
\[ q_j^n = \frac{(-1)^n }{(n!)^2 { 2 \alpha_j \choose n } } (\ad E)^n q_j .\]
\end{lem}
Let us denote 
\[ J^F_k := \{(i,m)\in J^F\times \ZZ_{\geq 0}| q_i^m \in \g(k) \text{ or }  q^i_m \in \g(-k)  \}\]
and consider the following diffential algebra homomorphisms
\begin{equation*}
\begin{aligned}
& \pi: P (\p)= P(\g^F) \otimes P([E,\g_{\leq -1/2}]) \to P (\g^F), \quad A\otimes B \mapsto A, \\
& \rho : P(\g) \to P(\p), \qquad a\mapsto \pi_{\leq \frac{1}{2}} (a) + (F|a) \text{ for } a\in \g, \\
\end{aligned}
\end{equation*}
where  $P(\mathcal{E})=S(\CC[\partial]\otimes \mathcal{E})$  for a vector superspace $\mathcal{E}$ and $\pi_{\leq 1/2}: \g \to \g_{\leq 1/2}$ is the projection map.

\begin{lem}\cite{DKV16_1} \label{Lem:generators} \hfill
\begin{enumerate}
\item If $(i,m) \in J^F_{t_1}$ and $(j,n)\in J^F_{t_2}$ for $t_1, t_2 \in \frac{\ZZ}{2}$ then 
\begin{equation}
\rho\{ q^i_{m\ \lambda\ } q^n_j\}=
\left\{ 
\begin{array}{ll}
\, 0 & \text{ if } {t_2}-{t_1}>1, \\
\, \delta_{i,j} \delta_{n, m+1} &  \text{ if } {t_2}-{t_1}=1, \\
\, [q_m^i, q_j^n] + \delta_{i,j} \delta_{m,n} k \lambda  & \text{ if } {t_2}-{t_1}\leq \frac{1}{2}.
\end{array}
\right.
\end{equation}
\item 
If $\gamma\in P(\g^F) \big( \CC[\partial][E, \g_{\leq -1/2}] \big)$ satisfies 
\[ \pi \rho \{a_\lambda \gamma\}=0 \quad \text{ for any } \quad  a\in \g_{\geq 1/2}\]
then $\gamma=0$.
\end{enumerate}
\end{lem}

By Lemma \ref{Lem:generators}, Proposition \ref{Prop:structure of W} and Remark \ref{Remark:structure of W}, we have the unique differential algebra isomorphism 
\begin{equation} \label{omega}
 \omega: P (\g^F) \to \mathcal{W}(\g,F), \qquad a\mapsto \omega_a \text{ for } a\in \g^F 
\end{equation}
such that $\omega_a= a+\gamma(a)+ \gamma^{\geq 2}(a)$ where 
\begin{equation} \label{decomp}
 \gamma(a)\in P(\g^F) \otimes \big( \CC[\partial]\otimes [E, \g_{\leq -1/2}] \big), \quad \gamma^{n}(a)\in P(\g^F) \otimes (\CC[\partial]\otimes [E, \g_{\leq -1/2}])^{\otimes n}
\end{equation}
 and $\gamma^{\geq 2}(a)= \sum_{i\geq 2} \gamma^i(a)$.\\
 
 Let us define the partial order $\prec$ on $ \large( \,  \bigcup_{k\in \frac{\ZZ}{2}} J_k^F \large) \cup \frac{\ZZ}{2}$ by 
 \begin{itemize}
\item $
( j_t, n_t)\prec (j_{t+1}, n_{t+1}) \ \text{ if and only if } \ 
\left\{ \begin{array}{l} \alpha_{t+1} -\alpha_t \geq 1 \text{ where } \\
(j_t, n_t)\in J^F_{\alpha_t} \text{ and } (j_{t+1}, n_{t+1})\in J^F_{\alpha_{t+1}}, \end{array}\right.$
\item $( j, n) \prec k$ (resp. $k \prec (j,n)$)   if and only if  $(j, n)\in J^F_{\beta}$ for $\beta\leq k-1$ (resp. $k\leq \beta-1$)
\end{itemize}
for $k \in \frac{\ZZ}{2}$. 

\begin{thm}
Let $\pi_{\g^F}:\g \to \g^F$ be the projection map and denote 
 $a^\sharp := \pi_{\g^F}(a)$ for $a\in \g$. For $q\in \g^F\cap \g(-\alpha)$, we have $\omega(q)=q+ \gamma(q)+ \gamma^{\geq 2}(q)$ where
\begin{equation}
\begin{aligned}
 \gamma(q)=  \sum_{p\in \ZZ_{\geq  0} } \gamma_p(q)= \sum_{p\in \ZZ_{\geq 0}} & \sum_{\substack{-\alpha-1 \prec ( j_0, n_0) \prec \cdots \\ \ \cdots \prec (j_{p-1}, n_{p-1}) \prec ( j_p, n_p) \prec \frac{1}{2}}}  s(j_0)\left([q, q_{n_0}^{j_0}] ^\sharp- (q| q_{n_0}^{j_0})k\partial \right)
  \\
 & \qquad  \prod_{ t=1, \cdots, p}   s(j_t)\left([q_{j_{t-1}}^{n_{t-1}+1}, q_{n_t}^{j_t}] ^\sharp- (q_{j_{t-1}}^{n_{t-1}+1}| q_{n_t}^{j_t})k\partial \right) \  q_{j_p}^{n_p+1}.
 \end{aligned}
\end{equation}
Note that $\displaystyle \prod_{t=1,\cdots , p}A_t:= A_1 A_2 \cdots A_p$.

\end{thm}

\begin{proof}

For $q_m^i\in \g_{>0}$, we have 
\[ \pi \rho (\{q_m^i{}_\lambda q\}) = -s(q,i)([q, q_m^i]^\sharp- (q|q_m^i) k \lambda)\]
and 
\begin{equation*}
\begin{aligned}
& \pi\rho (\{q_m^i{}_\lambda \gamma_p(q)\}) \\
& \\
& =   \sum_{\substack{-\alpha-1 \prec ( j_0, n_0) \prec \cdots \\ \ \cdots \prec ( j_p, n_p) \prec \frac{1}{2}}} s(q,i) s(q, j_p)    s(j_0)\left([q, q_{n_0}^{j_0}] ^\sharp- (q| q_{n_0}^{j_0})k(\partial +\lambda)\right)  \\
&    \hskip 3cm  \prod_{ t=1, \cdots, p}  s(j_t)\left([q_{j_{t-1}}^{n_{t-1}+1}, q_{n_t}^{j_t}] ^\sharp- (q_{j_{t-1}}^{n_{t-1}+1}| q_{n_t}^{j_t})k (\partial+\lambda) \right)  \cdot \pi\rho \{q_m^i{}_\lambda q_{j_p}^{n_p+1}\}  \\
&     \\
& = \sum_{\substack{(j_0, n_0) \prec ( j_1, n_1) \prec \cdots \\ \ \cdots \prec ( j_p, n_p) = (i,m) }} S_1 \cdot 
\left([q, q_{n_0}^{j_0}] ^\sharp- (q| q_{n_0}^{j_0})k(\partial +\lambda)\right) 
\\
&\hskip 1cm \prod_{ t=1, \cdots, p-1}  \left([q_{j_{t-1}}^{n_{t-1}+1}, q_{n_t}^{j_t}] ^\sharp- (q_{j_{t-1}}^{n_{t-1}+1}| q_{n_t}^{j_t})k(\lambda+\partial) \right)  \left([q_{j_{p-1}}^{n_{p-1}+1}, q_{m}^{i}] ^\sharp- (q_{j_{p-1}}^{n_{p-1}+1}| q_{m}^{i})k\lambda \right)\\
& \\ 
& +  \sum_{\substack{(j_0, n_0) \prec ( j_1, n_1) \prec \cdots \\ \ \cdots \prec ( j_p, n_p)\prec(i,m) }} S_2 \cdot  \left([q, q_{n_0}^{j_0}] ^\sharp- (q| q_{n_0}^{j_0})k(\partial +\lambda)\right) 
\\
&\hskip 1cm \prod_{ t=1, \cdots, p}  \left([q_{j_{t-1}}^{n_{t-1}+1}, q_{n_t}^{j_t}] ^\sharp- (q_{j_{t-1}}^{n_{t-1}+1}| q_{n_t}^{j_t})k(\lambda+\partial) \right)  \left([q_{j_{p-1}}^{n_{p-1}+1}, q_{m}^{i}] ^\sharp- (q_{j_{p-1}}^{n_{p-1}+1}| q_{m}^{i})k\lambda \right)\\
\end{aligned}
\end{equation*}
where $S_1:=s(q,i)s(j_0)s(j_1)\cdots s( j_{p-1})$ and $S_2:=-s(q,i)s(j_0)s(j_1)\cdots s(j_{p-1})s(j_p).$
Hence $\pi\rho (\{q_m^i{}_\lambda q+\gamma(q)\})=0.$ Since we know $\pi\rho (\{q_m^i{}_\lambda \gamma^{\geq 2}(q)\})=0$ and  by Lemma \ref{Lem:generators} (2),  we have proved the theorem.
\end{proof}

\begin{lem} \label{lem:PVA structure}
For any $t\in \frac{1}{2}\ZZ$, we have
\begin{equation} \label{lemma_PVAstructure}
\sum_{(j,n)\in J_{-t}^F}s(j) \, q_n^j \otimes q_j^{n+1}=-\sum_{(i,m)\in J_{t-1}^F} \, q_i^{m+1} \otimes q_m^i 
\end{equation}
\end{lem}

\begin{proof}
Note that 
\begin{itemize}
\item the both sides of \eqref{lemma_PVAstructure} are in $[E,\g(t-1)]\otimes [E, \g(-t)]$,
\item  $[E,\g(t)]$ is nondegenerately paired with $[F,\g(-t)]$.
\end{itemize}
Hence it suffices to show that for $a\in \g(-t+1)$ and $b\in \g(t)$, 
\[ \sum_{(j,n)\in J_{-t}^F}s(j) ([F,a]| q_n^j) ([F,b]|q_j^{n+1}) =-\sum_{(i,m)\in J_{t-1}^F} ([F,a]| q_i^{m+1}) 
([F,b]|q_m^i).  \]
One can check that 
\begin{equation*}
\begin{aligned}
 & \sum_{(j,n)\in J_{-t}^F}s(j) ([F,a]| q_n^j) ([F,b]|q_j^{n+1})= \sum_{(j,n)\in J_{-t}^F}- (q_{n+1}^j |a)([F,b]|q_j^{n+1}) = -([F,b]|a), \\
& \sum_{(i,m)\in J_{t-1}^F}([F,a]|q_i^{m+1})([F,b]|q_m^i)=- s(b)([F,a]|b)=([F,b]|a).
\end{aligned}
\end{equation*}
By the  skew-symmetry of  Lie superalgebras and the invariance of the bilinear form, we have  proved the lemma.
\end{proof}

\begin{thm}
Let us take $a\in \g(-t_1)\cap \g^F$ and $b\in \g(-t_2)\cap \g^F$. Then 
\begin{equation*}
\begin{aligned}
& \{ \omega(a){}_\lambda \omega(b)\}  = [a,b] + k\lambda(a|b)  \\
&  \hskip 1cm -s(a,b) \sum_{p\in \ZZ_{\geq 0}}\sum_{\substack{\, -t_2-1 \prec (j_0, n_0) \prec \cdots \\ 
\cdots \prec (j_{p}, n_{p})\prec t_1}}   s(j_0)s(j_1)\cdots s(j_p) \big(\omega([b, q_{n_0}^{j_0}]^\sharp)-(b| q_{n_0}^{j_0}) k(\lambda+\partial) \big)\\
&   \hskip 1cm   \left[ \prod_{t=1}^p \omega([ q_{j_{t-1}}^{n_{t-1}+1}, q_{n_t}^{j_t} ]^\sharp) - (q_{j_{t-1}}^{n_{t-1}+1}| q_{n_t}^{j_t})k(\lambda+\partial)\right]  \big(\omega([q_{j_{p}}^{n_{p}+1}, a]^\sharp)-(q_{j_{p}}^{n_{p}+1}|a)k\lambda \big).
\end{aligned}
\end{equation*}
\end{thm}

\begin{proof}
Let us denote 
\begin{equation*}
\begin{aligned}
&  \gamma_p(a)=  \sum_{\substack{-t_1-1 \prec (i_0, m_0) \prec \\ 
\cdots \prec (i_p, m_p)}} 
s(i_0)s(i_1)\cdots s(i_p)
 \big([a, q_{m_0}^{i_0}]^\sharp-(a| q_{m_0}^{i_0})k  \partial \big) \\
& \big([ q_{i_0}^{m_0+1}, q_{m_1}^{i_1}]^\sharp-(q_{i_0}^{m_0+1}| q_{m_1}^{i_1})k\partial )
\cdots  \big( [q_{i_{p-1}}^{m_{p-1}+1}, q_{m_p}^{i_p}]^\sharp-(q_{i_{p-1}}^{m_{p-1}+1}| q_{m_p}^{i_p})k\partial )q^{m_p+1}_{i_p}
\end{aligned}
\end{equation*}
and 
\begin{equation*}
\begin{aligned}
&  \gamma_s(b)=  \sum_{\substack{-t_2-1 \prec (j_0, n_0) \prec  \\ 
\cdots  \prec (j_s, n_s)}} 
s(j_0)s(j_1)\cdots s(j_s)
\big([b, q_{n_0}^{j_0}]^\sharp-(b| q_{n_0}^{j_0}) k \partial \big) \\
& \big([ q_{j_0}^{n_0+1}, q_{n_1}^{j_1}]^\sharp-(q_{j_0}^{n_0+1}| q_{n_1}^{j_1})k\partial )
\cdots \big([ q_{j_{s-1}}^{n_{s-1}+1}, q_{n_s}^{j_s}]^\sharp-(q_{j_{s-1}}^{n_{s-1}+1}| q_{n_s}^{j_s})k\partial )q^{m_s+1}_{i_s}.
\end{aligned}
\end{equation*}
To simplify the notation, let us denote 
\[ \gamma_p(a)= \sum_{  (i_p, m_p)\succ -t_1+p}A_{(i_p, m_p)}(\partial) q_{i_p}^{m_p+1}.\]
Then 
\begin{equation}\label{Eqn:PVA structure 1}
\begin{aligned}
& \pi \rho \{\gamma_p(a){}_\lambda \gamma_s(b)\} \\
& = \sum_{\substack{-t_1-1 \prec (i_0, m_0)   \\ 
\cdots  \prec (i_p, m_p)}} \sum_{\substack{-t_2-1 \prec (j_0, n_0)  \\ 
\cdots \prec (j_s, n_s)}} \big( s(i_0)s(i_1)\cdots s(i_p) \big) \big( s(j_0)s(j_1)\cdots s(j_s)\big)\\
& \hskip 2cm  s(a,b)s(a,j_s) \  B_{(j_s, n_s)}(\lambda+\partial)\  \pi \rho \left( \{ \ A_{(i_p, m_p)}(\partial)q_{i_p}^{m_p+1} {}_{\lambda} \ q_{j_s}^{n_s+1} \ \}\right),
\end{aligned}
\end{equation}
where
\begin{equation}\label{Eqn:PVA structure 3}
\begin{aligned}
&  B_{(j_s, n_s)}(\lambda+\partial)   = \sum_{\substack{-t_2-1 \prec (j_0, n_0)  \prec \cdots \prec (j_s, n_s) }}  \big([b, q_{n_0}^{j_0}]^\sharp-(b| q_{n_0}^{j_0}) k (\lambda+\partial) \big) \\
 & \hskip 1cm  ([ q_{j_0}^{n_0+1}, q_{n_1}^{j_1}]^\sharp-(q_{j_0}^{n_0+1}| q_{n_1}^{j_1})k(\lambda+\partial) ) \cdots ( [ q_{j_{s-1}}^{n_{s-1}+1}, q_{n_s}^{j_s}]^\sharp-(q_{j_{s-1}}^{n_{s-1}+1}| q_{n_s}^{j_s})k(\lambda+\partial)).
\end{aligned}
\end{equation} \\
Observe that
\begin{equation} \label{Eqn:PVA structure 2}
\begin{aligned}
& \sum_{  (i_p, m_p)\succ -t_1+p} \pi \rho \left( \{ A_{(i_p, m_p)}(\partial)q_{i_p}^{m_p+1} {}_{\lambda} \ q_{j_s}^{n_s+1} \} \right) \\
& = \sum_{\substack{-t_1-1 \prec (i_0, m_0) \prec \cdots \prec (i_p, m_p) \prec \frac{1}{2}}}s(a q_{i_0}, q_{i_0}q_{j_s}) s(q_{i_0}q_{i_1}, q_{i_1}q_{j_s})\cdots s(q_{i_{p-1}}q_{i_p}, q_{i_p}q_{j_s})\  \\
& \hskip 2cm \pi \rho \left(\{\ q_{i_p}^{m_p+1}\ {}_\lambda \ q_{j_s}^{n_s+1} \} \right)  \big([q_{i_{p-1}}^{m_{p-1}+1}, q_{m_p}^{i_p}]^\sharp+(q_{i_{p-1}}^{m_{p-1}+1}| q_{m_p}^{i_p})k(\lambda+\partial) \big) \cdots  \\
& \hskip 2cm \cdots \big([q_{i_0}^{m_0+1}, q_{m_1}^{i_1}]^\sharp+(q_{i_0}^{m_0+1}| q_{m_1}^{i_1})k(\lambda+\partial) \big)\big([a, q_{m_0}^{i_0}]^\sharp+(a| q_{m_0}^{i_0})k\lambda \big) \\
& \\
& =\sum_{\substack{-t_1-1 \prec (i_0, m_0) \prec  
\cdots \prec (i_p, m_p) \prec \frac{1}{2}}} \quad  (-1)^{p} s(a, j_s)  s(i_0)s(i_1)\cdots s(i_p) \\
& \hskip 1cm  \pi \rho \left(\{q_{j_s}^{n_s+1}\ {}_{-\lambda-\partial}\ q_{i_p}^{m_p+1}\} \right) \big([q_{m_p}^{i_p}, q_{i_{p-1}}^{m_{p-1}+1}]^\sharp-(q_{m_p}^{i_p}|q_{i_{p-1}}^{m_{p-1}+1})k(\lambda+\partial) \big) \cdots  \\
& \hskip 1cm \cdots \big([q_{m_1}^{i_1}, q_{i_0}^{m_0+1}]^\sharp-( q_{m_1}^{i_1}| q_{i_0}^{m_0+1})k(\lambda+\partial) \big)\big([ q_{m_0}^{i_0},a]^\sharp-( q_{m_0}^{i_0}|a)k\lambda \big). \\
\end{aligned}
\end{equation}
By Lemma \ref{lem:PVA structure}, 
\begin{equation}\label{Eqn:PVA structure 4}
\begin{aligned}
&\eqref{Eqn:PVA structure 2} = \sum_{ -\frac{1}{2} \preceq  (j_{s+1}, n_{s+1}) \prec \cdots \prec (j_{p+s+1}, n_{p+s+1}) \prec t_1} \quad  - s(a, j_s)\\
&   \pi \rho \left(\{q_{j_s}^{n_s+1}\ {}_{-\lambda-\partial}\ q_{n_{s+1}}^{j_{s+1}}\} \right) \big([q_{j_{s+1}}^{n_{s+1}+1}, q_{n_{s+2}}^{j_{s+2}}]^\sharp-(q_{j_{s+1}}^{n_{s+1}+1}|q_{n_{s+2}}^{j_{s+2}})k(\lambda+\partial) \big) \cdots  \\
&  \big([q_{j_{s+p}}^{n_{s+p}+1}, q_{n_{s+p+1}}^{j_{s+p+1}}]^\sharp-( q_{j_{s+p}}^{n_{s+p}+1}| q_{n_{s+p+1}}^{j_{s+p+1}})k(\lambda+\partial) \big)\big([ q_{j_{s+p+1}}^{n_{s+p+1}+1},a]^\sharp-( q_{j_{s+p+1}}^{n_{s+p+1}+1}|a)k\lambda \big), \\
\end{aligned}
\end{equation}
where $\alpha\preceq (j, n)$ (resp, $\alpha\succeq(j,n)$) if and only if  $q_j^n \in \g_{\geq \alpha}$ (resp. $q_j^n \in \g_{\leq \alpha}$) for $k\in \frac{\ZZ}{2}$. \\

Now, by    \eqref{Eqn:PVA structure 4}, we have 
\begin{equation*}\label{Eqn:PVA structure 5}
\pi\rho (\{q_{j_s}^{n_s+1}{}_{-\lambda-\partial}\, q^{j_{s+1}}_{n_{s+1}}\})= 
\left\{ \begin{array}{ll} [q_{j_s}^{n_s+1}, q^{j_{s+1}}_{n_{s+1}}]^\sharp-(q_{j_s}^{n_s+1}| q^{j_{s+1}}_{n_{s+1}})k(\lambda+\partial) & \text{ if }  (j_s, n_s) \prec (j_{s+1}, n_{s+1})\\ 
(-1)^{j_s+1}   & \text{ if } \left( \begin{array}{l}(j_s,n_s)=( j_{s+1}, n_{s+1}), \\ \text{ and } q_{j_s}^{ n_s}\in \g_{-1/2} \end{array}\right) \\
0 & \text{ otherwise}.
\end{array}\right.
\end{equation*}
Hence 
\begin{equation*}
\begin{aligned}
&  \pi\rho(\{ \gamma_p(a){}_\lambda \gamma_s(b)\})= -s(a,b)(P_{p,s}+Q_{p,s}+R_{p,s}+S_{p,s}) \\       
\end{aligned}
\end{equation*}
for $p,s\in \ZZ_{\geq 0}$, where 
\begin{equation} \label{p}
\begin{aligned}
P_{p,s} = & \sum_{\substack{(j_s, n_s) \in J^F_{-1/2},\\
-t_2 \preceq (j_{0}, n_{0}) \prec \cdots  \prec (j_{s+p+1}, n_{s+p+1})\prec t_1 }} s(j_0)s(j_1)\cdots s(j_{s+p+1})\\
&\hskip 1cm \big([b, q_{n_{0}}^{j_{0}}]^\sharp-(b| q_{n_{0}}^{j_{0}})k (\lambda+\partial) \big) \big([q_{j_{0}}^{n_{0}+1}, q_{n_{1}}^{j_{1}}]^\sharp-(q_{j_{0}}^{n_{0}+1}| q_{n_{1}}^{j_{1}})k (\lambda+\partial) \big) \\
 &  \big([q_{j_{1}}^{n_{1}+1}, q_{n_{2}}^{j_{2}}]^\sharp-(q_{j_{1}}^{n_{1}+1}| q_{n_{2}}^{j_{2}}) k(\lambda+\partial) \big)\cdots  \big( [ q_{j_{s+p+1}}^{n_{s+p+1}+1}, a]^\sharp-(q_{j_{s+p+1}}^{n_{s+p+1}+1}| a)k(\lambda)),
\end{aligned}
\end{equation}
\begin{equation}\label{q}
\begin{aligned}
Q_{p,s} = &\sum_{\substack{(j_{s+1}, n_{s+1}) \in J^F_{-1/2},\\
-t_2 \preceq (j_{0}, n_{0}) \prec \cdots  \prec (j_{s+p+1}, n_{s+p+1})\prec t_1 }}   s(j_0)s(j_1)\cdots s(j_{s+p+1})\\
&\hskip 1cm \big([b, q_{n_{0}}^{j_{0}}]^\sharp-(b| q_{n_{0}}^{j_{0}})k (\lambda+\partial) \big) \big([q_{j_{0}}^{n_{0}+1}, q_{n_{1}}^{j_{1}}]^\sharp-(q_{j_{0}}^{n_{0}+1}| q_{n_{1}}^{j_{1}}) k(\lambda+\partial) \big) \\
 &  \big([q_{j_{1}}^{n_{1}+1}, q_{n_{2}}^{j_{2}}]^\sharp-(q_{j_{1}}^{n_{1}+1}| q_{n_{2}}^{j_{2}})k (\lambda+\partial) \big)\cdots  \big( [ q_{j_{s+p+1}}^{n_{s+p+1}+1}, a]^\sharp-(q_{j_{s+p+1}}^{n_{s+p+1}+1}| a)k(\lambda)),
\end{aligned}
\end{equation}
\begin{equation}\label{r}
\begin{aligned}
R_{p,s}= &\sum_{\substack{(j_s, n_s), (j_{s+1}, n_{s+1}) \not \in J^F_{-1/2},\\
-t_2 \preceq (j_{0}, n_{0}) \prec \cdots  \prec (j_{s+p+1}, n_{s+p+1})\prec t_1 }}    s(j_0)s(j_1)\cdots s(j_{s+p+1})\\
&\hskip 1cm \big([b, q_{n_{0}}^{j_{0}}]^\sharp-(b| q_{n_{0}}^{j_{0}}) k (\lambda+\partial) \big) \big([q_{j_{0}}^{n_{0}+1}, q_{n_{1}}^{j_{1}}]^\sharp-(q_{j_{0}}^{n_{0}+1}| q_{n_{1}}^{j_{1}})k (\lambda+\partial) \big) \\
 &  \big([q_{j_{1}}^{n_{1}+1}, q_{n_{2}}^{j_{2}}]^\sharp-(q_{j_{1}}^{n_{1}+1}| q_{n_{2}}^{j_{2}})k (\lambda+\partial) \big)\cdots  \big( [ q_{j_{s+p+1}}^{n_{s+p+1}+1}, a]^\sharp-(q_{j_{s+p+1}}^{n_{s+p+1}+1}| a)k(\lambda)),
\end{aligned}
\end{equation}
 and 
\begin{equation*}\label{s}
\begin{aligned}
& S_{p,s} = \sum_{\substack{(j_s, n_s), (j_{s+1}, n_{s+1})  \in J^F_{-1/2},\\
-t_2 \preceq (j_{0}, n_{0}) \prec \cdots \prec (j_s, n_s)  \\ \prec (j_{s+2}, n_{s+2}) \prec \cdots \prec (j_{s+p+1}, n_{s+p+1})\prec t_1 }}     -\big( s(j_0)s(j_1)\cdots s(j_{s-1})\big) \big( s(j_{s+1})s(j_{s+2})\cdots s(j_{s+p+1})\big)\\
&\hskip 2cm \big([b, q_{n_{0}}^{j_{0}}]^\sharp-(b| q_{n_{0}}^{j_{0}}) k(\lambda+\partial) \big) \big([q_{j_{0}}^{n_{0}+1}, q_{n_{1}}^{j_{1}}]^\sharp-(q_{j_{0}}^{n_{0}+1}| q_{n_{1}}^{j_{1}}) k(\lambda+\partial) \big) \cdots \\
 &  \hskip 1cm \big([q_{j_{s-1}}^{n_{s-1}+1}, q_{n_{s}}^{j_{s}}]^\sharp-(q_{j_{s-1}}^{n_{s-1}+1}| q_{n_{s}}^{j_{s}}) k(\lambda+\partial) \big)
  \big([q_{j_{s+1}}^{n_{s+1}+1}, q_{n_{s+2}}^{j_{s+2}}]^\sharp-(q_{j_{s+1}}^{n_{s+1}+1}| q_{n_{s+2}}^{j_{s+2}}) k(\lambda+\partial) \big)\\
 &  \hskip 5cm \cdots  \big( [ q_{j_{s+p+1}}^{n_{s+p+1}+1}, a]^\sharp-(q_{j_{s+p+1}}^{n_{s+p+1}+1}| a)k(\lambda)).
\end{aligned}
\end{equation*}
In addition, we have 
\begin{equation}
\begin{aligned}
&  \pi\rho(\{ a\ {}_\lambda \gamma_s(b)\})= -s(a,b)(P_{-1,s}+R_{-1,s})               \\
&  \pi\rho(\{ \gamma_p(a) \ {}_\lambda \ b\})= -s(a,b)(Q_{p,-1}+R_{p,-1})  
\end{aligned}
\end{equation}
where $P_{-1, s}, R_{-1, s}, Q_{p,-1}, R_{p,-1}$ are obtained by letting $p=-1$ or $s=-1$ in \eqref{p}, \eqref{q} or \eqref{r}.
Observe
\[  P_{p,s}= -S_{p+1,s} = Q_{p+1, s-1} \]
and thus
\begin{equation*}
\begin{aligned}
 \sum_{N=0}^\infty\sum_{p+s=N} &  \pi\rho \{ a+\gamma_p(a) \, {}_{\lambda} \, b+ \gamma_s(b)\}= [a,b]+k\lambda(a|b) \\
& -s(a,b) \sum_{\substack{-k \preceq (j_{0}, n_{0}) \prec \cdots \\ 
\cdots  \prec (j_{s+p+1}, n_{s+p+1})\prec h }} s(j_0)s(j_1)\cdots s(j_{s+p+1})\\
&\hskip 1cm \big([b, q_{n_{0}}^{j_{0}}]^\sharp-(b| q_{n_{0}}^{j_{0}}) k(\lambda+\partial) \big) \big([q_{j_{0}}^{n_{0}+1}, q_{n_{1}}^{j_{1}}]^\sharp-(q_{j_{0}}^{n_{0}+1}| q_{n_{1}}^{j_{1}}) k(\lambda+\partial) \big) \cdots \\
 &  \big([q_{j_{s+p}}^{n_{s+p}+1}, q_{n_{s+p+1}}^{j_{s+p+1}}]^\sharp-(q_{j_{s+p}}^{n_{s+p}+1}| q_{n_{s+p+1}}^{j_{s+p+1}})k(\lambda+\partial) \big)  \big( [ q_{j_{s+p+1}}^{n_{s+p+1}+1}, a]^\sharp-(q_{j_{s+p+1}}^{n_{s+p+1}+1}| a)k(\lambda)),
\end{aligned}
\end{equation*}
which implies the theorem.
\end{proof}

\vskip 3cm

\section{Supersymmetric Poisson vertex algebras} \label{Sec:SUSY PVA}

In this section, we recall the notion of supersymmetric vertex algebras. For more details, we refer to \cite{HK06}.

\begin{defn} Let $\mathcal{R}$ be a Lie superalgebra with an odd operator $D:\mathcal{R}\to \mathcal{R},$ that is  $D(\mathcal{R}_{\bar{i}}) \subset \mathcal{R}_{\bar{i}+\bar{1}}$ for $i=0,1$.
\begin{enumerate}
\item A $\chi$-bracket on a $\CC[D]$-module $\mathcal{R}$ is a linear map 
\[ [\ {}_\chi \ ] : \mathcal{R} \otimes \mathcal{R}  \to \CC[\chi ] \otimes \mathcal{R} \]
where $\chi$ is an odd indeterminate and $p(a)+p(b)+1= p([a{}_\chi b])$ for any homogeneous elements $a,b\in \mathcal{R}$. In this paper, we usually denote 
\[ \textstyle [\, a \, {}_\chi  \, b\, ]= \sum_{n\in \ZZ_{\geq0}} \chi^n a_{[n]}b \]
for $a_{[n]}b\in \mathcal{R}$.

\item A $\CC[D]$-module $\mathcal{R}$ is called a {\it supersymmetric (SUSY) Lie conformal algebra (LCA)} if it satisfies the following properties: 
\begin{itemize}
\item (sesquilinearity) $[Da{}_\chi b]= \chi[a {}_\chi b]$, $[a{}_\chi Db]= -s(a) (D+\chi) [a{}_\chi b]$,
\item (skew-symmetry) $[a{}_\chi b] = s(a,b)[b_{-\chi-D} a]$ where
\[ \textstyle [b{}_{-\chi-D}a] = \sum_{n\in \ZZ_{\geq 0}} (-D-\chi)^n b_{[n]}a \]
and $\CC[\chi]\otimes \mathcal{R}$ is a $\CC[D]$-module via
\begin{equation}
\chi D + D \chi= -2 \chi^2.
\end{equation}
\item (Jacobi identity) $[a{}_\chi [b{}_\gamma c]]=-s(a) [[a{}_\chi b]{}_{\chi+\gamma} c]-s(a,b)s(a)s(b) [b{}_\gamma[a{}_\chi c]]$, where $\gamma$ is an odd indeterminate supercommuting with $\chi$. The LHS of the Jacobi identity is computed via 
\[ [ \, a \, {}_\chi \,  \gamma^n b_{[n]}c \, ]= (-s(a)\gamma)^n [\, a\, {}_\chi \,  b_{[n]}c\, ] \]
and the RHS is computed via 
\begin{equation*}
\begin{aligned}
&  [ \, \chi^n a_{[n]}b \, {}_{\chi+\gamma} \, c\, ] = (-\chi)^n [\, a_{[n]}b\, {}_{\chi+\gamma}\,  c\, ],\\
&   [\, b\, {}_\gamma  \,  \chi^n a_{[n]}c\, ]= (-s(b)\chi)^n [\, b\, {}_\gamma  \, a_{[n]}c\, ].
\end{aligned}
\end{equation*}
\end{itemize}
\end{enumerate}
\end{defn}

\begin{defn}
A tuple $(\mathcal{P}, \{ \, {}_\chi \, \}, 1, \cdot, D)$ is called a {\it SUSY PVA} if it satisfies: 
\begin{itemize}
\item $(\mathcal{P}, 1,\, \cdot \, , D)$ is a unital differential supercommutative algebra with  odd derivation $D$,
\item $(\mathcal{P}, \{\, {}_\chi \, \}, D)$ is a SUSY PVA,
\item (Leibniz rule) $\{a{}_\chi bc\} = \{a{}_\chi b\} c + s(b,c) \{a{}_\chi c\} b$ for any $a,b,c\in \mathcal{P}$.
\end{itemize} 
\end{defn}

The (right) Leibniz rule and skew-symmetry of SUSY LCAs induce the left Leibniz rule
\begin{equation}
\{ab{}_\chi c\} = s(b,c) \{a {}_{\chi+D} c\}_{\to} b+ s(a,bc) \{b{}_{\chi+D}c\}_{\to} a,
\end{equation} 
where 
\[ \{a {}_{\chi+D} c\}_{\to} b = \sum_{n\in \ZZ_{\geq 0}} s(ac)^n a_{[n]}c (\chi+D)^n b. \]

\begin{prop}
A SUSY PVA is a PVA. More precisely, if $(\mathcal{P},  \{ \, {}_\chi \, \}, 1, \cdot, D)$ is a SUSY PVA then $(\mathcal{P},  \{ \, {}_\lambda \, \}, 1, \cdot, \partial=D^2)$ is a PVA via
\begin{equation} \label{PVA-SUSYPVA}
 a_{(n)}b:=(-1)^n a_{[2n+1]}b, 
\end{equation}
where $\{a {}_\lambda b\}= \sum_{n\in \ZZ_{\geq 0}} \lambda^n a_{(n)}b$ and $\{a {}_\chi b\} = \sum_{n\in \ZZ_{\geq 0}} \chi^n a_{[n]}b$ for $a,b\in \mathcal{P}$.
\end{prop}

\begin{proof}
By the sesquilineaity of SUSY LCA $\mathcal{P}$,
\[ D^2 a_{[2n+1]} b = a_{[2n-1]}b \quad \text{and }  \quad a_{[2n+1]}D^2 b = D^2(a_{[2n+1]} b)-a_{[2n-1]}b. \]
for $a,b\in \mathcal{P}$ and $n\in \ZZ_{\geq 1}$. Thus, \eqref{PVA-SUSYPVA} implies that 
\[ \partial a_{(n)}b = -a_{(n-1)}b \quad \text{ and }  \quad a_{(n)} \partial b = \partial (a_{(n)}b) + a_{(n-1)} b,\]
which are equivalent to the sesquilinearity of LCA.

By the skew-symmetry of SUSY LCA $\mathcal{P}$,
\begin{equation} \label{skew-sym, SUSY-nonSUSY}
 \textstyle  a_{[2n+1]}b = \sum_{m\geq n}s(a,b) (-1)^{m+1}{m \choose n} D^{2m-2n} ( b_{[2m+1]} a)
\end{equation}
for $a,b\in \mathcal{P}$ and $n\in \ZZ_{\geq 0}$. 
More precisely, the LHS and RHS are the coefficients of $\chi^{2n+1}$ in $\{a{}_\chi b\}$ and $s(a,b) \{b{}_{-\chi-D} a\}$, respectively. Due to \eqref{PVA-SUSYPVA} and $D^2=\partial$,  \eqref{skew-sym, SUSY-nonSUSY} is equivalent to
\begin{equation} \label{skew-sym, SUSY-nonSUSY_2}
 \textstyle a_{(n)}b = -s(a,b) \sum_{m\geq n} {m\choose n} (-1)^m \partial^{m-n} (b_{(m)}a) .
\end{equation}
One can check that the LHS and RHS of \eqref{skew-sym, SUSY-nonSUSY_2} are coefficients of $\lambda^n$ in $\{a{}_\lambda b\}$ and $-s(a,b) \{b{}_{-\lambda-\partial} a\}$. 

To see the Jacobi identity of $\mathcal{P}$ as a LCA, one has to show that 
\begin{equation} \label{Jacobi_LCA_second}
\textstyle a_{(n)}(b_{(m)}c) = \sum_{j\geq m} {j \choose m} (a_{(n+m-j)}b)_{(j)}c +s(a,b) b_{(m)} (a_{(n)}c)
\end{equation}
for $a,b,c\in \mathcal{P}$ and $n,m\in \ZZ_{\geq 0}$. In fact, \eqref{Jacobi_LCA_second} is the coefficient of $\chi^{2n+1} \gamma^{2m+1}$ in 
\[ [a{}_\chi[b{}_\gamma c]]= -s(a)[[a{}_\chi b]_{\chi+\gamma} c]-s(a,b)s(a)s(b) [b{}_\gamma[a_{\chi}c]].\]
Hence the Jacobi identity of LCA $\mathcal{P}$ holds.

Finally, a similar argument works for the right Leibniz rule of $\mathcal{P}$. Hence the SUSY PVA structure of $\mathcal{P}$ naturally induces the PVA structure of $\mathcal{P}$. 
\end{proof}

\begin{prop} \cite{CS19} \label{Prop:VS to SUSY PVA}  \hfill
\begin{enumerate}
\item Let $V$ be a vector superspace and $\mathcal{R}(V):= \CC[D]\otimes V$ be the $\CC[D]$-module freely generated by $V$. Let
\begin{equation} \label{chi-bracket on V} 
[ \ {}_\chi \ ] : V \otimes V \to \CC[\chi]\otimes (\mathcal{R}(V))
\end{equation}
be an odd linear map satisfying the skew-symmetry and Jacobi identity of the $\chi$-brackets. Then the bracket \eqref{chi-bracket on V}  can be extended to the bracket on $\mathcal{R}(V)$ via the sesquilinearity  and $\mathcal{R}(V)$ endowed with the $\chi$-bracket is a SUSY PVA.
\item Let $(\mathcal{R}\, ,\, [ \ {}_\chi \ ])$ be a SUSY LCA. Then the supersymmetric algebra $\mathcal{P}:=S(\mathcal{R})$ endowed with the $\chi$-bracket $\{\ {}_\chi \ \}$  extended from the $\chi$-bracket on $\mathcal{R}$ via the Leibniz rule is a SUSY PVA.
\end{enumerate}
\end{prop}

Let us fix the vector superspace $V$ with a basis $\mathcal{B}=\{u_i|i\in I\}$, which is homogenous with respect to the parity and let 
\[ \mathcal{R}:= \CC[D]\otimes V, \quad \mathcal{P}:= S(\mathcal{R}).\]
Then $\mathcal{R}$ is the vector superspace with the basis $\{\, u_i^{[m]}:= D^m(u_i) \,|\, i\in I, m\in \ZZ_{\geq 0}\}$ and $\mathcal{P}$ is the superalgebra of polynomials in $\{\, u_i^{[m]}\,  |\, i\in I, m\in \ZZ_{\geq 0}\, \}$.
For simplicity of notation, we write $s(i):= s(u_i)$, $s(i,j):=s(u_i, u_j)$ for $i,j\in I$. 

\begin{defn}
Let $\frac{\partial}{\partial u_i^{[m]}}$ be the derivation on $\mathcal{P}$ of parity $p(u_i)+\bar{m}\in \ZZ/2\ZZ$ such that 
\[ \frac{\partial}{ \partial u_i^{[m]}}(u_j^{[n]})= \delta_{i,j} \delta_{m,n}.\]
\end{defn}

One can check the commutator $\big[ \frac{\partial}{\partial u_i^{[m]}}, D\big]$ is (i) a derivation of parity $p(u_i)+m+1$ and  
(ii) $\big[\frac{\partial}{ \partial u_i^{[m]}},D\big](u_j^{[n]})= \delta_{i,j} \delta_{m-1,n}.$
Since such derivation is unique, we have  
\[ \textstyle  \big[\frac{\partial}{ \partial u_i^{[m]}},D\big]=  \frac{\partial}{ \partial u_i^{[m-1]}}, \textstyle \]
where $\frac{\partial}{\partial u_i^{[-1]}}:=0$.

\begin{thm}[Master formula]  \cite{CS19}  
Let $\mathcal{P}$ be a SUSY PVA. Then for $a,b\in \mathcal{P}$, 
\begin{equation}
\begin{aligned}
 \{a{}_\chi b\} &  = \sum_{i,j\in I, \ m,n\in \ZZ_{\geq 0}} S(a_{(i,m)}, b_{(j,n)})\, b_{(j,n)} \{u_i^{[m]}{}_{\chi+D} u_j^{[n]} \}_{\to} a_{(i,m)} \\
& = \sum_{i,j\in I, \ m,n\in \ZZ_{\geq 0}} S(a_{(i,m)}, b_{(j,n)}) (-1)^{n+mn+\frac{(m+1)m}{2}} s(i)^{n+m}s(j)^m \\
& \hskip 3cm b_{(j,n)} (\chi+D)^n \{ u_i \, {}_{\chi+D} \, u_j\}_{\to} (\chi+D)^m a_{(i,m)} 
\end{aligned}
\end{equation}
where 
\begin{itemize}
\item $a_{(i,m)}:= \frac{\partial}{\partial{u_i^{[m]}}} a$\ and \ $b_{(j,n)}:= \frac{\partial}{\partial{u_j^{[n]}}}b,$
\item $S(a_{(i,m)}, b_{(j,n)}) := s(b_{(j,n)}) \, s(b_{(j,n)}, u_j^{[n]} a) \,  s(a_{(i,m)}, u_j^{[n]}).$
\end{itemize}
\end{thm}

\begin{ex} [Affine SUSY PVA] \label{Ex:affine SUSY}
Let $\g$ be a Lie superalgebra with an even invariant supersymmetric bilinear form $(\, |\, )$ and let $\bar{\g}=\{ \bar{a}| a\in \g\}$ be the vector superspace such that $p(a)=1+p(\bar{a})$ for a homogeneous element $a\in \g$.  Consider the $\chi$-bracket 
\begin{equation}\label{chi_affine}
 [ \ {}_\chi \ ] \, : \,  (\bar{\g} \oplus \CC K)  \, \otimes \, (\bar{\g} \oplus \CC K) \, \to  \, \CC[\chi]\otimes (\bar{\g} \oplus \CC K)
\end{equation}
such that
\begin{equation} 
[\bar{a}{}_\chi \bar{b}] =s(a) (\overline{[a,b]} + \chi K  (a|b)), \quad [K {}_\chi \overline{a}]=[K{}_\chi K]=0 
\end{equation} 
for $a,b,\in \g$. 

Let $\mathcal{R}:= \CC[D]\otimes \bar{\g} \oplus \CC K$ be a $\CC[D]$-module via $D(D^n \bar{a})=D^{n+1} \bar{a}$ and  $D(K)=0.$ The bracket \eqref{chi_affine} can be extended to the bracket on $\mathcal{R}$ using the sesquilinearity. Thus, $\mathcal{R}$ is a SUSY LCA called the affine SUSY LCA associated with $\g$. 

By Proposition \ref{Prop:VS to SUSY PVA}, the supercommutative differential algebra $ \mathcal{P}_K(\bar{\g})=S(\mathcal{R})$ is a SUSY PVA endowed with the bracket $\{\ {}_\chi \ \}$ induced from the bracket on $\mathcal{R}$ and the Leibniz rule.
In addition, the supersymmetric differential algebra $\mathcal{P}(\bar{\g}):= S(\CC[D] \otimes \bar{\g})\simeq P_K / (K-k) P_K$ is also a SUSY PVA called the affine SUSY PVA associated with $\bar{\g}$ and $k\in \CC$. 
\end{ex}

\section{SUSY classical W-algebras} \label{Sec:SUSY_W}

\subsection{First definition of SUSY classical W-algebras} \label{Subsec: SUSY W_first}

Let  $\g$ be a simple finite Lie superalgebra with a subalgebra $\mathfrak{s}=\text{Span}_{\CC}\{ E,e,H,f,F \}$ isomorphic to $\mathfrak{osp}(1|2)$ such that 
\begin{itemize}
\item $(E,H,F)$ is an $\mathfrak{sl}_2$-triple,
\item $e,f$ are odd elements,
\item $[H,e]=e$, $[H,f]=-f$, $[e,e]= 2E$, $[f,f]=-2F$, $[e,f]=-H$, $[F,e]=f$, $[E,f]=e$.
\end{itemize}

We also assume that $\g$ is endowed with an even nondegenerate invariant supersymmetric bilinear form $(\ | \ )$ such that
\[(H|H)= 2(E|F)=2, \quad (e|f)=-(f|e)=-2.\] 
Since $H$ is in $\mathfrak{sl}_2$, the Lie superalgebra $\g$ can be decomposed into $\g=\bigoplus_{i\in \ZZ/2} \g(i)$ where $\g(i) = \{a\in \g| [H,a]= 2i a \}$.

Recall the affine SUSY PVA $\mathcal{P}(\bar{\g})$ in Example \ref{Ex:affine SUSY} and consider the super Lie subalgebra 
\[ \textstyle \n = \g_{>0}. \]
Let $\mathcal{I}_{f}$ be the differential algebra ideal of $\mathcal{P}(\bar{\g})$ generated by $\{ \bar{n}-(f|n)|n\in \n \}$. Then the quotient space $\mathcal{P}(\bar{\g})/\mathcal{I}_f$ is a well-defined differential algebra  and let $\text{ad}_\chi \bar{\n} : \mathcal{P}(\bar{\g})/\mathcal{I}_f \to \CC[\chi] \otimes \mathcal{P}(\bar{\g})/\mathcal{I}_f$ be  defined by 
\begin{equation} \label{Eqn:chi_n action}
 \textstyle \text{ad}_\chi \bar{n}( [A])= \sum_{m\in \ZZ_{\geq 0}}\chi^m [ \bar{n}_{[m]}A ]\in \CC[\chi] \otimes  \big( \mathcal{P}(\bar{\g})/\mathcal{I}_f \big)
\end{equation}
where $[A]$ and $[\bar{n}_{[m]}A]$ are the images of $A$ and $\bar{n}_{[m]}A$ in $\mathcal{P}(\bar{\g})/\mathcal{I}_f$, respectively. Since $\{\bar{n}{}_\chi \mathcal{I}_f \} \in \CC[\chi] \otimes  \mathcal{I}_f$, the vector superspace 
\begin{equation}
\WW(\bar{\g},f):= (\mathcal{P}(\bar{\g})/\mathcal{I}_f)^{\text{ad}_\chi \bar{\n}} = \{ \bar{A} \in \mathcal{P}(\bar{\g})/\mathcal{I}_f | \text{ad}_\chi (\bar{A})=0\}
\end{equation}
is well-defined.

\begin{prop}
The vector superspace $\WW(\bar{\g},f)$ is a SUSY PVA endowed with the bracket induced from the $\chi$-bracket of the affine SUSY PVA $\mathcal{P}(\bar{\g})$. 
\end{prop}

\begin{proof}
Let us take $A,B\in \mathcal{P}(\bar{\g})$ such that $[A], [B] \in \WW(\bar{\g},f)$ and $n\in \n$.  
To show that $\WW(\bar{\g},f)$ is a differential superalgebra, it is sufficient to show that  $\{\bar{n}{}_\chi AB\}$ and $\{ \bar{n} {}_\chi DA\}$ are in $\CC[\chi]\otimes \mathcal{I}_f$. Since 
\begin{equation} \label{Eqn:algebra_W}
\begin{aligned}
& \{\bar{n}{}_\chi AB\} = \{\bar{n}{}_\chi A\} B + s(A,B) \{\bar{n}{}_\chi B\} A \\
& \{\bar{n}{}_\chi DA\}= s(n) (D+\chi) \{\bar{n}{}_\chi A\}
\end{aligned}
\end{equation}
and $\{\bar{n}{}_\chi A\}, \ \{ \bar{n}{}_\chi B\} \in \CC[\chi] \otimes \mathcal{I}_f$,  both equations in \eqref{Eqn:algebra_W} are in $\CC[\chi] \otimes \mathcal{I}_f$.
In order to show that $\WW(\bar{\g},f)$ is SUSY LCA, we need to show that $\{ \bar{n} {}_\chi \{ A{}_\gamma B\}\} \in \CC[\chi, \gamma] \otimes \mathcal{I}_f.$  Since $\{ I{}_\chi A\}$ and $\{I{}_\chi B\}$ are in $\CC[\chi] \otimes \mathcal{I}_f$ for any $I\in \mathcal{I}_f$, we have
\[ \{\bar{n}{}_\chi \{A {}_\gamma B\}\} = s(n) \{ \{ \bar{n}{}_\chi A\}_{\chi+\gamma} B\}  +s(A,n)s(n) \{A_\gamma \{ \bar{n}_\chi B\}\}\in \CC[\chi, \gamma]\otimes \mathcal{I}_f.\]
Hence $\WW(\bar{\g},f)$ is a SUSY PVA.
\end{proof}

\begin{defn} \label{Def: first def/SUSY W}
The SUSY PVA $\WW(\bar{\g},f)$ is called the SUSY classical W-algebra associated with $\bar{\g}$ and $f$. 
\end{defn}

\subsection{Second definition of SUSY classical W-algebras}

In this section, we introduce SUSY classical W-algebras via SUSY classical BRST complexes. We refer to \cite{MR, MRS19} for the construction of SUSY W-algebras using SUSY quantum BRST complexes. The SUSY classical BRST complexes introduced in this section can be understood as classical limits of the quantum complexes.

\vskip 2mm

Consider the following two SUSY PVAs:

\begin{enumerate}[(I)]
\item Let $j_{\bar{\g}}=\{ \,  j_{\bar{a}} \,  |\,  a\in \g\}$ be isomorphic to $\bar{\g}$ as vector superspaces.
Consider the supersymmetric algebra 
 \[S(\CC[D] \otimes j_{\bar{\g}})\]
generated by  $\CC[D]\otimes j_{\bar{\g}}.$ Define the SUSY PVA structure on $S(\CC[D] \otimes j_{\bar{\g}})$ by
\[ \, [ \ j_{\bar{a}_{\ \chi\ }} j_{\bar{b}} \ ]= s(a,\bar{b})j_{\overline{[a,b]}}+ k\chi(a|b) \ \text{ for } a,b\in \g, \ k\in \CC. \] 
According to Proposition \ref{Prop:VS to SUSY PVA}, $S(\CC[D] \otimes j_{\bar{\g}})$ is a SUSY PVA.
\item  For the subspace $\n=\bigoplus_{i>0} \g(i)$, let
 $\phi_{\n} \simeq \n\subset \g$ and $\phi^{\bar{\n}_-}\simeq \bar{\n}_-\subset \bar{\g}$ as vector superspaces. The supersymmetric algebra
 \[ S(\CC[D]\otimes (\phi_{\n}\oplus \phi^{\overline{\n}}))\]
endowed with the $\chi$-bracket defined by
 \[ \,  [ \ \phi^{\bar{a}}_{\ \ \chi \ }  \phi_{b}\ ]= [ \ \phi_{b \ \chi \ } \phi^{\overline{a}}\ ]=(a|b)\]
 is a SUSY PVA.
 \end{enumerate}
 
 \begin{rem}
 The SUSY PVA (I) is isomorphic to the affine SUSY PVA $\mathcal{P}(\bar{\g})$ by the map \[j_{\bar{a}}\mapsto   \left\{ \begin{array}{ll} \mathrm{i}\, \bar{a} & \text{ if } a\in \g \text{ is odd, } \\ \bar{a} & \text{ if } a\in \g \text{ is even, } \end{array}\right.\] for  the imaginary number $\mathrm{i}\in \CC$.
 \end{rem}
 
Consider the SUSY PVA  
\begin{equation}\label{Eqn:comp}
\mathcal{C}(\bar{\g},  f )= S(\CC[D]\otimes j_{\bar{\g}}) \otimes S(\CC[D]\otimes (\phi_{\n}\oplus \phi^{\bar{\n}}))
\end{equation}
which is the  tensor product of two SUSY PVAs (I) and (II). Let us denote  the index set by $I_>$ such that  
\begin{equation}
\{u_\alpha\}_{\alpha\in I_>}  \text{ and }    \{u^\alpha\}_{\alpha\in I_>}
\end{equation}
are bases of $\n$ and $\n_-= \g_{<0}$, respectively, satisfying $(u^\alpha| u_\beta)=\delta_{\alpha, \beta}.$ For simplicity of notation, let 
\[ s(\alpha,\beta)= s(u_\alpha,u_\beta), \quad s(\beta)= s(u_\beta)\]
for $\alpha, \beta \in I_>$ and let 
\begin{equation}
j_{\bar{\alpha}}= j_{\overline{u}_\alpha}, \quad \phi_\alpha= \phi_{u_\alpha}, \quad \phi^{\bar{\alpha}}= \phi^{\bar{u}^\alpha}, \quad \phi_{a}= \phi_{\pi_+ (a)}, \quad  \phi^{\bar{a}}= \phi^{\overline{\pi_- ({a})}}
\end{equation}
for the projection maps $\pi_+: \g \to \n$, $\pi_-= \g \to \n_-$.

 Take the element 
\begin{equation} \label{BRST:d}
 d^c= \sum_{\alpha\in I_>} (j_{\bar{\alpha}}-c( f |u_\alpha))\phi^{\bar{\alpha}}+\frac{1}{2}\sum_{\alpha,\beta\in I_>} s(\alpha,\beta) s(\beta)  \phi_{[u_{\alpha}, u_{\beta}]} \phi^{\bar{\beta}} \phi^{\bar{\alpha}} \in \mathcal{C}(\bar{\g},  f ),
 \end{equation}
for $c\in \CC$. If there is no need to emphasize the value $c\in \CC$, we simply denote $d^c$ as $d$. 

\begin{prop} \label{Prop:differential} \hfill
\begin{enumerate}
\item The $\chi$-brackets between $d$ and the elements in $\mathcal{C}(\bar{\g}, f )$ are
\begin{equation} \label{Eqn:4.4_0405}
\begin{aligned}
& \ \{d_{\ \chi\ } j_{\bar{a}}\}= \sum_{\alpha\in I_>} s(\alpha,a)s(\alpha) \phi^{\bar{\alpha}} j_{\overline{[ u_\alpha, a]}}-\sum_{\alpha\in I_>} s(\alpha)k(\chi+D)\phi^{\bar{\alpha}}(u_\alpha|a), \\
& \ \{ d_{\ \chi \ } \phi^{\bar{\alpha}}\}= \frac{1}{2} \sum_{\beta\in I_>} s(\alpha,\beta)s(\beta)\phi^{\bar{\beta}} \phi^{\overline{[u_\beta,u^\alpha]}}, \\
& \ \{  d_{\ \chi \ } \phi_\alpha\}=  -s(\alpha)  j_{\bar{\alpha}} - c ( f |u_\alpha)+ \sum_{\beta\in I_>} s(\alpha,\beta)s(\beta)\phi^{\bar{\beta}} \phi_{[u_\beta, u_\alpha] }.
\end{aligned}
\end{equation}
\item Recall that $d_{[0]} A = \{ d_\chi A\}|_{\chi=0}$ for $A\in \mathcal{C}(\bar{\g},  f )$. We have   $d_{[0]}^2=0.$
\end{enumerate}
\end{prop}

\begin{proof}
(1) We have 
\begin{equation}
\begin{aligned}
 \{ d_\chi j_{\overline{a}}\} & = \textstyle \sum_{\alpha\in I_>} \{(j_{\bar{\alpha}}-c( f |u_\alpha))\phi^{\bar{\alpha}}_{\ \  \chi \ } j_{\bar{a}}\}  =  \textstyle \sum_{\alpha\in I_>} \{j_{\bar{a}\ \ -\chi-D\ }   (j_{\bar{\alpha}}-c( f |u_\alpha))\phi^{\bar{\alpha}}\}  \\ 
 & =\textstyle \sum_{\alpha\in I_>} s(a,\alpha)s(a) ( j_{\overline{[a, u_\alpha]}}+k(a|u_\alpha) (-\chi-D) ) \phi^{\bar{\alpha}} \\
 & = \textstyle \sum_{\alpha\in I_>}\left( s(\alpha, a)s(\alpha)  \phi^{\bar{\alpha}} j_{\overline{[u_\alpha, a]}} -s(\alpha) k(\chi+D)\phi^{\bar{\alpha}} (u_\alpha|a)\right).
\end{aligned}
\end{equation}
Here we used the skew-symmetry and Leibniz rule of the $\chi$-bracket and the fact that $d$ is an even element. We can also check that  
\begin{equation}
\begin{aligned}
\{ d_{\ \chi \ } \phi^{\bar{\alpha}}\} & =   \textstyle \frac{1}{2} \sum_{\beta,\gamma\in I_>} \left\{ \phi^{\bar{\alpha}}_{\ \ -\chi-D\ }  s(\beta, \gamma) s(\gamma)  \phi_{[u_{\beta}, u_{\gamma}]} \phi^{\bar{\gamma}} \phi^{\bar{\beta}} \right\} \\ 
& =   \textstyle \frac{1}{2}  \sum_{\beta,\gamma\in I_>} s(\beta, \gamma) s(\gamma)(u^\alpha| [u_\beta, u_\gamma]) \phi^{\bar{\gamma}} \phi^{\bar{\beta}} =   \frac{1}{2}  \sum_{\alpha,\beta\in I_>} s(\alpha,\beta) s(\beta)\phi^{\bar{\beta}} \phi^{\overline{[u_\beta, u^\alpha]}}
\end{aligned} 
\end{equation}
and 
\begin{equation}
\begin{aligned}
\{d_{\ \chi \ } \phi_\alpha \}& =\{ \phi_{\alpha\  -\chi-D\ } d\} \\
& = -s(\alpha)\overline{u}_{\alpha}-c( f |u_\alpha) + \textstyle \frac{1}{2}\sum_{\beta,\gamma\in I_>} \left\{\phi_{\alpha\ -\chi-D\ }s(\beta,\gamma)s(\gamma)  \phi_{[u_{\beta}, u_{\gamma}]} \phi^{\bar{\gamma}} \phi^{\bar{\beta}} \right\}\\
& =   \textstyle -s(\alpha) \overline{u}_\alpha - c( f |u_\alpha)+ \sum_{\beta\in I_>}s(\alpha,\beta)s(\beta)\phi^{\bar{\beta}} \phi_{[u_\beta, u_\alpha]}.
\end{aligned}
\end{equation}

(2)  By the Jacobi identity, 
\[ d_{[0]}(d_{[0]}A) = -((d_{[0]}d)_{[0]} A) -d_{[0]}(d_{[0]}A)\]
for any $A \in \mathcal{C}(\bar{\g}, f, c)$.  Hence it is sufficient to show $d_{[0]}d=0$. By (1), one can verify that 
\begin{itemize}
\item $ \textstyle \sum_{\alpha, \beta\in I_>} \{ \, j_{\bar{\alpha}} \phi^{\bar{\alpha}}\, {}_\chi \, j_{\bar{\beta}} \phi^{\bar{\beta}} \, \} = \sum_{\alpha, \beta\in I_>} s(\alpha, \beta) s(\alpha) j_{\overline{[u_\alpha, u_\beta]}} \phi^{\bar{\beta}} \phi^{\bar{\alpha}},$
\item 
$\textstyle  \sum_{\alpha, \beta, \gamma \in I_>} \{ \,  j_{\bar{\alpha}}\phi^{\bar{\alpha}}\, {}_\chi \,  s(\beta, \gamma) s(\gamma) \phi_{[u_\beta, u_\gamma]} \phi^{\bar{\gamma}} \phi^{\bar{\beta}} \, \}
= \sum_{\beta, \gamma \in I_>}-s(\beta, \gamma) s(\beta) j_{\overline{[u_\beta, u_\gamma]}} \phi^{\bar{\gamma}} \phi^{\bar{\beta}}.$
\end{itemize}
Hence, for $d_2:=  \sum_{\alpha, \beta\in I_>} s(\alpha,\beta) s(\beta)  \phi_{[u_{\alpha}, u_{\beta}]} \phi^{\bar{\beta}} \phi^{\bar{\alpha}}$, we have 
\[\textstyle  \{ d{}_\chi d\} = \frac{1}{4}\{ d_2 {}_\chi d_2\} . \]
By the master formula, 
\begin{equation*}
\begin{aligned}
 \{ d_2 {}_\chi d_2\} &  = \textstyle \sum_{\gamma, \delta \in I_> } \left( \{d_2 {}_\chi s(\gamma, \delta) s(\delta) \phi_{[u_\gamma, u_\delta]} \} \phi^{\bar{\delta}} \phi^{\bar{\gamma}} + 2 \{d_2 {}_\chi s(\gamma) s(\delta) \phi^{\bar{\delta}} \} \phi_{[u_\gamma, u_\delta]} \phi^{\bar{\gamma}}\right).
\end{aligned}
\end{equation*}
Since 
\begin{itemize}
\item $\textstyle \sum_{\gamma, \delta \in I_> } \{d_2 {}_\chi s(\gamma, \delta) s(\delta) \phi_{[u_\gamma, u_\delta]} \} \phi^{\bar{\delta}}
 \phi^{\bar{\gamma}}=\sum_{\alpha, \gamma, \delta \in I_> }s(\alpha) s(\delta) \phi_{[u_\alpha, [u_\gamma, u_\delta]]} \phi^{\bar{\alpha}} \phi^{\bar{\gamma}} \phi^{\bar{\delta}}, $ 
\item $\textstyle \sum_{\gamma, \delta \in I_> } \{d_2 {}_\chi s(\gamma) s(\delta) \phi^{\bar{\delta}} \} \phi_{[u_\gamma, u_\delta]} \phi^{\bar{\gamma}}= -\frac{1}{2} \sum_{\alpha, \gamma, \delta \in I_> }s(\alpha) s(\delta) \phi_{[u_\alpha, [u_\gamma, u_\delta]]} \phi^{\bar{\alpha}} \phi^{\bar{\gamma}} \phi^{\bar{\delta}},$
\end{itemize} 
we have $\{ d_2 {}_\chi d_2\} =0$. Thus $\{d {}_\chi d \}=0$ and $d_{[0]}^2=0$. 
\end{proof}

By Proposition \ref{Prop:differential} (2), the map $d_{[0]}$ is an odd differential  on $\mathcal{C}(\bar{\g},  f )$.

\begin{prop} \label{Prop:W_1-PVA}\hfill
\begin{enumerate}
\item  If $A\in \mathcal{C}(\bar{\g},  f )$ is in the image of $d_{[0]}$ then $DA$ is also in the image of $d_{[0]}$.
\item If $B\in \ker d_{[0]}$ then $ d_{[0]}(A) B$ for $A\in \mathcal{C}(\bar{\g},  f )$  is in the image of $d_{[0]}$.
\item  If $B\in \ker d_{[0]}$ and $A\in \mathcal{C}(\bar{\g},  f ,c)$ then both $\{d_{[0]}(A)_\chi B\}$ and $\{B_\chi d_{[0]}(A)\}$ are in $\CC[\chi]\otimes \text{im}(d_{[0]}).$
\end{enumerate}
\end{prop}

\begin{proof}
(1) By the sesquilinearity of SUSY PVAs, $d_{[0]}(DA)=-D(d_{[0]} A).$  \\
(2) By the Leibniz rule of SUSY PVAs, \[d_{[0]}(A B)= s(A) A\,  d_{[0]}(B)+ s(A,B)s(B) Bd_{[0]}(A)=d_{[0]}(A) B.\]
(3) By the Jacobi identity of SUSY PVAs, 
\begin{equation*}
\begin{aligned}
& d_{[0]}(\{A_\chi B\})= -\{d_{[0]}(A)_\chi B\} -s(A)\{A_\chi d_{[0]}(B)\}=-\{d_{[0]}(A)_\chi B\}, \\
& d_{[0]}(\{B_\chi A\})= -\{d_{[0]}(B)_\chi A\} -s(B)\{B_\chi d_{[0]}(A)\}=-s(B)\{B_\chi d_{[0]}(A)\}.
\end{aligned}
\end{equation*}
Hence we have proved the proposition.
\end{proof}

\begin{thm} \label{Thm: SISU W_ BRST}
The SUSY PVA structure of  $\mathcal{C}(\bar{\g},  f )$ induces the SUSY PVA structure on 
\begin{equation}\label{Eqn:SUSY W BRST}
 \WW^{\ c}_{\text{BR}}(\bar{\g}, f):= H(\mathcal{C}(\bar{\g},f), d^c_{[0]}).
\end{equation}
 \end{thm}
 
 \begin{proof}
 By Proposition \ref{Prop:W_1-PVA} (1), (2) and (3),  $\WW_{\text{BR}}^{\ c}(\bar{\g}, f )$ is closed under derivation $D$, the supercommutative product and  the $\chi$-bracket. 
 \end{proof}

We also simply denote the algebra \eqref{Eqn:SUSY W BRST} by $\WW_{\text{BR}}(\bar{\g}, f )$ if the constant $c$ does not play a  crucial role.

\subsection{Equivalence of two definitions of SUSY classical W-algebras}

In this section, we compare SUSY PVAs defined in Theorem \ref{Thm: SISU W_ BRST} and Definition \ref{Def: first def/SUSY W}. 
We  refer to \cite{DK} for the detailed properties of bigraded complexes adopted in this section.

Recall the complex $\mathcal{C}(\bar{\g},  f )$ in (\ref{Eqn:comp}) and the differential $d_{[0]}$ for $d$ in (\ref{BRST:d}). 
In order to see the structure of $\WW_{\text{BR}}(\bar{\g}, f)$, we consider the building block 
\begin{equation}\label{Building_Block}
 J_{\bar{a}}= j_{\bar{a}} -\sum_{\beta\in I_>}s(a,\beta)s(a)s(\beta)\phi^{\bar{\beta}} \phi_{[u_\beta, a]} \in \mathcal{C}(\bar{\g},  f ) 
\end{equation}
for $a\in \g$. 
Then 
\begin{equation}
\begin{aligned}
\{d_\chi  J_{\overline{a}}\} & = \sum_{\gamma\in I_>} \left( s(a,\gamma)s(\gamma) \phi^{\bar{\gamma}} j_{\overline{[u_\gamma,a]} }-ks(\gamma)D\phi^{\bar{\gamma}}(u_\gamma|a)  \right) \\
& \hskip 2cm  -\sum_{\beta\in I_>}s(a,\beta)s(a)s(\beta) \{ d_\chi  \phi^{\bar{\beta}} \phi_{[u_\beta, a]}\}.
\end{aligned}
\end{equation}
In the last term of the previous equation, 
\begin{equation}\label{Eqn:4.9}
\begin{aligned}
    \sum_{\beta\in I_>}   \{ d_\chi \phi^{\bar{\beta}} \phi_{[u_\beta, a]}\} & =       \frac{1}{2}\  \sum_{\beta, \gamma\in I_>} s(\gamma,\beta)s(\gamma) \phi^{\bar{\gamma}} \phi^{\overline{[u_\gamma, u^\beta]}}\phi_{[u_\beta, a]} \\
  &  \quad  + \sum_{\beta\in I_>} \phi^{\bar{\beta}}\left((-1)^{a} \overline{\pi_+[u_\beta, a]} -(-1)^{a}c(f|\pi_+[u_\beta, a])\right) \\
  &   \quad    -\sum_{\beta, \gamma\in I_>} s(\gamma, [u_\beta,a])s(\gamma) s(\beta)\phi^{\bar{\beta}}\phi^{\bar{\gamma}} \phi_{[u_\gamma,\pi_+ [u_\beta, a]]}.
\end{aligned}
\end{equation}
 Since 
\begin{equation}
\begin{aligned}
&     \sum_{\beta,\gamma\in I_>}s(\gamma,\beta)s(\gamma)\phi^{\bar{\gamma}}  \phi^{\overline{[u_\gamma, u^\beta]}}\phi_{[u_\beta, a]}= \sum_{\beta, \gamma, \delta\in I_>}s(\gamma,\delta)\phi^{\bar{\gamma}} \phi^{\bar{\delta}} ([u_\gamma, u^\beta]|u_\delta) \phi_{[u_\beta, a]}\\
= &    \sum_{\gamma, \delta\in I_>} \big( s(\gamma,\delta)s(\delta)  \phi^{\bar{\delta}} \phi^{\bar{\gamma}} \phi_{[u_\gamma,[u_\delta, a]]}+ s(\gamma,\delta)s(\gamma) \phi^{\bar{\gamma}} \phi^{\bar{\delta}} \phi_{[u_\delta,[u_\gamma, a]]} \big)\\
= &   \sum_{\gamma, \delta\in I_>}2 \   s(\gamma,\delta)s(\delta)\phi^{\bar{\delta}} \phi^{\bar{\gamma }} \phi_{[u_\gamma,[u_\delta, a]]}, 
\end{aligned}
\end{equation}
we have
\begin{equation} \label{Eqn:5.4_0405}
\begin{aligned}
& \{d_\chi  J_{\bar{a}}\}  \\
 = & \sum_{\beta\in I_>} \big( s(a,\beta)s(\beta)\phi^{\bar{\beta}} j_{\overline{\pi_{\leq 0} [u_\beta, a]}}-s(\beta) k(D+\chi)\phi^{\bar{\beta}}(u_\beta|a)+s(a,\beta)s(\beta) \phi^{\bar{\beta}}c( f |[u_\beta,a])\big)\\
&- \sum_{\beta, \gamma\in I_>} s(a,\beta)s(a,\gamma)s(\beta,\gamma)s(a)s(\gamma)\phi^{\bar{\beta}} \phi^{\bar{\gamma}} \phi_{\overline{[u_\gamma,\pi_{\leq0} [u_\beta, a]]}}\\
= &  \sum_{\beta\in I_>}s(a,\beta)s(\beta) \phi^{\bar{\beta}} \big(J_{\overline{\pi_{\leq 0}[u_\beta, a]}}+c ( f |[u_\beta, a])\big)-\sum_{\beta\in I_>}s(\beta)k D\phi^{\bar{\beta}}(u_\beta|a),
\end{aligned}
\end{equation}
where $\pi_{\leq 0}: \g \to  \g_{\leq 0}$ is the projection map. 

The $\chi$-bracket between two building blocks is
\begin{equation}
\begin{aligned}
& \{ J_{\bar{a}\ \chi \, } J_{\bar{b}}\}  =\{j_{\bar{a}_\chi } j_{\bar{b}}\} +\sum_{\beta, \gamma\in I_>} \left\{  s(a,\beta)s(a)s(\beta)\phi^{\bar{\beta}} \phi_{[u_\beta, a]\ \chi \, } s(b,\gamma)s(b)s(\gamma)\phi^{\bar{\gamma}} \phi_{[u_\gamma,b]}\right\}\\
& =\{j_{\bar{a}_\chi } j_{\bar{b}}\} -\sum_{\beta,\gamma\in I_>} 
s(a,\beta)s(a,b)s(b,\beta)s(b)s(\beta)
\left( \phi^{\bar{\beta}}\phi_{[\pi_+[u_\beta, a], b]}-s(a,b)\phi^{\bar{\beta}} \phi_{[\pi_+[u_\beta, b], a]} \right)\\
& =  s(a,b)s(a)j_{\overline{[a,b]} }+ k(D+\chi)(a|b) -\sum_{\beta\in I_>}s(a,b) s(a,\beta)s(b,\beta)s(b)s(\beta)\phi^{\bar{\beta}}\phi_{[u_\beta,[a,b]]}  \\
&+ \sum_{\beta,\gamma\in I_>}  s(a,\beta)s(a,b)s(b,\beta)s(b)s(\beta)
\left( \phi^{\bar{\beta}}\phi_{[\pi_{\leq 0}[u_\beta, a], b]}-s(a,b)\phi^{\bar{\beta}} \phi_{[\pi_{\leq 0}[u_\beta, b], a]} \right).
\end{aligned}
\end{equation}
Hence 
\begin{equation} \label{Eqn:5.7_0409}
\{ J_{\overline{a}\, \chi  \, } J_{\overline{b}}\}= s(a,b)s(a) J_{\overline{[a,b]}} +k(\chi+D)(a|b)
\end{equation}
if $a,b\in \g_{\leq 0}$ or $a,b\in \g_{>0}$. Let  
\[ r_+= \phi_{\n} \oplus d_{[0]}(\phi_{\n}), \quad r_-= J_{\bar{\g}_{\leq 0}}\oplus \phi^{\overline{\n}_-}, \quad \mathcal{R}_+=\CC[D]\otimes  r_+ , \quad \mathcal{R}_-=  \CC[D]\otimes r_- .\]

\begin{lem}[K{\"u}nneth theorem] (e.g., see \cite{DK}) \label{Lem:Kunn}
Let $U_1$, $U_2$ and $U= U_1\otimes U_2$ be superalgebras and $d_1$ and $d_2$ be differentials on $U_1$ and $U_2$, respectively. If $ d=d_1\otimes 1 + 1\otimes d_2$ then 
\[ H(U,d)= H(U_1, d_1) \otimes H(U_2, d_2).\]
Further, if $d$ is a linear differential map on a vector superspace $V$ then 
\[ H(S(V), d)\simeq S(H(V,d))\]
where $S(V)$ is the supersymmetric algebra generated by $V$ and $d$ on $S(V)$ is the differential map induced from that on $V$.
\end{lem}

\begin{prop} \label{Prop:5.4_0405}
Recall that $\WW_{\text{BR}} (\bar{\g},  f)=H(\mathcal{C}(\bar{\g}, f ), d_{[0]})$.
\begin{enumerate}
\item $\mathcal{C}(\bar{\g}, f )= S(\mathcal{R}_+)\otimes S(\mathcal{R}_-)$.
\item $d_{[0]}|_{S(R_+)}\subset S(R_+)$ and $d_{[0]}|_{S(\mathcal{R}_-)}\subset S(\mathcal{R}_-)$.
\item $\WW_{\text{BR}}(\bar{\g},  f )\simeq H(S(\mathcal{R}_-), d_{[0]}=d_{[0]}|_{S(\mathcal{R}_-)})$.
\end{enumerate}
\end{prop}

\begin{proof}
(1)  It is enough to check that every element in $\g$, $\phi_{\n}$, $\phi^{\bar{\n}_-}$ are in $S(\mathcal{R}_+) \otimes S(\mathcal{R}_-)$.  \\
(2) Since $d_{[0]}(\phi_n)\in S(\mathcal{R}_+)$ and $d_{[0]}^2(\phi_n)=0$, we see that $d_{[0]}|_{S(\mathcal{R}_+)}\subset S(\mathcal{R}_+)$. By (\ref{Eqn:5.4_0405}) and (\ref{Eqn:4.4_0405}), we obtain $d_{[0]}|_{S(\mathcal{R}_-)}\subset S(\mathcal{R}_-)$.  \\
(3)  Using the K{\"u}nneth theorem, we have
\begin{equation*}
\begin{aligned}
 \WW_{\text{BR}}(\g,f,k) & \simeq H(S(\mathcal{R}_+), d_{[0]}|_{S(\mathcal{R}_+)})\otimes H(S(\mathcal{R}_-), d_{[0]}|_{S(\mathcal{R}_-)}) \\
 & \simeq H(S(\mathcal{R}_-), d_{[0]}|_{S(\mathcal{R}_-)}).
 \end{aligned}
 \end{equation*}
Here, we used $H(S(\mathcal{R}_+), d_{[0]}|_{S(\mathcal{R}_+)})=S(H(\mathcal{R}_+, d_{[0]}|_{\mathcal{R}_+}))=\CC.$
\end{proof}

Now we observe the cohomology $H(S(\mathcal{R}_-), d_{[0]})$ by defining $\frac{1}{2}\ZZ$-bigrading on $S(\mathcal{R}_-)$: 
\begin{equation} \label{Eqn:bigrade}
 \text{gr}(J_a)=( g_a, -g_a), \quad \text{gr}(\phi^\beta)=(-g_\beta+1/2, g_\beta+1/2) 
\end{equation}
where $a\in \g(g_a)$ and $u_\beta\in \g(g_\beta)$ and $\text{gr}(D)=(0,0)$ which induces the bigrading on $S(\mathcal{R}_-)$. We write  \[S(\mathcal{R}_-)^n:= \{\, A\in S(\mathcal{R}_-) \, | \,  \gr (A) = (p,q)  \text{ and } p+q=n   \} \] for $n \in \ZZ_{\geq 0}$. 
 Consider the decreasing filtration $\{ F^p S(\mathcal{R}_-) | p\in \ZZ/2\}$ such that 
\[ F^p S(\mathcal{R}_-) := \text{Span}_\CC \{ A \in S(\mathcal{R}_-)| \gr(A)=(p', q) \text{ for } p'\geq p \}.\] Then  $F^p S(\mathcal{R}_-)= \bigoplus_{n\in \ZZ}F^p S(\mathcal{R}_-)^n $ where  $F^p S(\mathcal{R}_-)^n  := F^p S(\mathcal{R}_-) \cap S(\mathcal{R}_-)^n$. 
One can check that 
\[ d_{[0]}|_{S(\mathcal{R}_-)}(F^p S(\mathcal{R}_-)^n ) \subset F^p S(\mathcal{R}_-)^{n+1}.\]
Thus the cohomology $H(S(\mathcal{R}_-), d_{[0]})$ is also $\ZZ_{\geq 0}$-graded and $\ZZ/2$-filtered space with 
\[ F^p H^n (S(\mathcal{R}_-), d_{[0]}) = \frac{\text{ker}(d_{[0]}|_{ F^p S(\mathcal{R}_-)^n}) }{ \text{im}(d_{[0]}|_{ F^p S(\mathcal{R}_-)^{n-1}})}\]
and let 
\[ \gr^{p,q} H(S(\mathcal{R}_-), d_{[0]}):= \frac{F^p H^{p+q} (S(\mathcal{R}_-), d_{[0]}) }{ F^{p+\frac{1}{2}}H^{p+q}(S(\mathcal{R}_-), d_{[0]})}.\]
Denote 
\[ \gr S(\mathcal{R}_-):= \bigoplus_{p,q\in \ZZ/2} \gr^{p,q} S(\mathcal{R}_-), \]
where $ \gr^{p,q} S(\mathcal{R}_-)= F^p S(\mathcal{R}_-)^{p+q} / F^{p+\frac{1}{2}} S(\mathcal{R}_-)^{p+q}$ and consider the differential  
 $d^{\gr}$  on $\gr S(\mathcal{R}_-)$ induced from $d_{[0]}$.  Then  
\begin{equation} \label{graded_diff}
 d^{\text{gr}}(J_a)=\sum_{\beta\in S} s(a,\beta)s(\beta)\phi^\beta ( f |[u_\beta, a]) , \quad d^{\text{gr}}(\phi^\beta)=0
\end{equation}
and the cohomology  $H(\gr S(\mathcal{R}_-), d^{\gr})$ is $\frac{\ZZ}{2}$-bigraded via 
\[ H^{p,q} (\gr S(\mathcal{R}_-), d^{\gr})= \frac{ \text{ker}( d^{\gr} : \text{gr}^{p,q} S(\mathcal{R}_-) \to \text{gr}^{p,q+1} S(\mathcal{R}_-) )}{\text{im}( d^{\gr} : \text{gr}^{p,q-1} S(\mathcal{R}_-) \to \text{gr}^{p,q} S(\mathcal{R}_-)  )}. \]

\begin{lem} \label{Lem:(R_-,d)good}
 The complex $(S(\mathcal{R}_-), d_{[0]})$ is good, i.e., $H^{p,q}(\gr S(\mathcal{R}_-, d^{\gr}), d^{\gr})=0$ for any $p,q \in \ZZ/2$ such that $p+q \neq 0$. Moreover, 
 \[ H(\gr S(\mathcal{R}_-), d^\gr)=  S(\CC[D]\otimes \bar{\g}^{ f }).\]
\end{lem}
\begin{proof}
There is a canonical isomorphism $\gr S(\mathcal{R}_-) \simeq S(\gr \mathcal{R}_-)$ that induces
\[ H(\gr S(\mathcal{R}_-), d^{\gr})\simeq H(S(\gr \mathcal{R}_-), d^{\gr}).\]
 By (\ref{graded_diff}), we have $H(\gr \mathcal{R}_-,d^\gr)= \bigoplus_{p\in \frac{\ZZ}{2}} H^{p,-p}(\gr \mathcal{R}_-,d^\gr) =\CC[D] \otimes \overline{\g}^{ f }$ and thus 
\begin{equation}
\textstyle H(\gr S(\mathcal{R}_-), d^\gr)=\bigoplus_{p\in \frac{\ZZ}{2}} H^{p,-p}(\gr S(\mathcal{R}_-), d^{\gr})=  S(\CC[D]\otimes \bar{\g}^{ f })
\end{equation}
where $\bar{\g}^{ f }= \{ \overline{a}\in\bar{\g}| [ f ,a]=0\}.$
\end{proof}

 Consider another grading $\Delta$ on $S(\mathcal{R}_-)$ defined by 
 \begin{equation}
\Delta_{J_{\bar{a}}}=\frac{1}{2}-g_a, \quad \Delta_{\phi^\beta}=g_\beta
\end{equation}
for $a\in \g(g_a)$ and $u_\beta \in \g(g_\beta)$. For $\delta\in \frac{ \NN}{2}$, denote 
\begin{equation}
S(\mathcal{R}_-)[\delta]=\{ A\in S(\mathcal{R}_-) | \Delta_A = \delta\}.
\end{equation}
One can check that 
(i) $S(\mathcal{R}_-)[\delta]$ is finite dimensional for any $\delta \in \frac{\NN}{2}$, 
(ii) $d_{[0]}|_{S(\mathcal{R}_-)[\delta]} \subset S(\mathcal{R}_-)[\delta]$.
By Lemma \ref{Lem:(R_-,d)good} and (i) and (ii) in the previous paragraph, $(S(\mathcal{R}_-)[\delta], d)$ is a good and locally finite complex.  

\begin{prop}\label{Prop:graded cohomology}
We have 
\[ \gr^{p,q} H(S(\mathcal{R}_-)[\delta], d_{[0]}) \simeq H^{p,q} (\gr S(\mathcal{R}_-)[\delta],d^{\gr})\]
for any $p\in \frac{\ZZ}{2}$. More precisely, 
\begin{equation} \label{key_cohomology}
\begin{aligned}
& F^p S(\mathcal{R}_-)^n[\delta] \cap \text{ker} \, d_{[0]} = F^p S(\mathcal{R}_-)^n[\delta] \cap \text{im} \, d_{[0]} \text{ for } n\neq 0,\\
& F^p S(\mathcal{R}_-)^0[\delta] \cap d_{[0]}^{-1} (F^{p+\frac{1}{2}} S(\mathcal{R}_-)^1[\delta])= F^{p+\frac{1}{2}} S(\mathcal{R}_-)^0[\delta] + F^p S(\mathcal{R}_-)^0 [\delta]\cap \text{ker}\,  d_{[0]}.
\end{aligned}
\end{equation}

\end{prop}

\begin{proof}
See Lemma 4.2 in \cite{DK}; the same proof applies.
\end{proof}

\begin{cor}
 As differential algebras, 
\[\WW_{\text{BR}}(\bar{\g}, f ,k)=H^0(S(\mathcal{R}_-), d_{[0]}|_{S(\mathcal{R}_-)}) \simeq S(\CC[D]\otimes\bar{\g}^f).\]
\end{cor}
\begin{proof}
This follows from Lemma \ref{Lem:(R_-,d)good} and Proposition \ref{Prop:graded cohomology}.
\end{proof}

The following proposition (Proposition \ref{Prop:5.5_0408}) is proved by the argument used in Lemma 4.13 in \cite{DK}.

\begin{prop} \label{Prop:5.5_0408}
The differential algebra 
  $\WW_{\text{BR}}(\bar{\g}, f)$ is freely generated by 
\[\{ \, E_i= J_{\bar{u}_i} +R_i  \, |\,  u_i\in \g^{ f } \cap \g(g_i), \ R_i \in  F^{g_i+\frac{1}{2}} S(\mathcal{R}_-)[1/2-g_i]\}_{i\in I},\]
where $\{u_i\}_{i\in I}$ is a basis of $\g^{ f }$. In other words, 
\[ \WW_{\text{BR}}(\bar{\g}, f) \simeq \CC[D^n E_i|i\in I, n\in \ZZ_{\geq 0}].\]
\end{prop}
\begin{proof}
In the second equality in \eqref{key_cohomology}, for any $i\in I$, $J_{\bar{u}_i}$ is an element in $F^{g_i} S(\mathcal{R}_-)^0 [1/2-g_i]\cap d_{[0]}^{-1}(F^{g_i+\frac{1}{2}} S(\mathcal{R}_-)^1)$. Hence there is an element 
\[ E_i \in F^{g_i}S(\mathcal{R}_-)^0[1/2-g_i] \cap \text{ker} \ d_{[0]}\]
such that $E_i= J_{\bar{u}_i} + e_i$ for some $e_i \in F^{g_i+\frac{1}{2}} S(\mathcal{R}_-)^0[1/2-g_i]$. It is enough to show that 
\[ \WW_{\text{BR}}(\bar{\g}, f)= S(\CC[D]\otimes_{\CC} \text{Span}_\CC \{ E_i\}).\]
Since $E_i \in \WW_{\text{BR}}(\bar{\g}, f)$, we have $\WW_{\text{BR}}(\bar{\g}, f)\supset S(\CC[D]\otimes_{\CC} \text{Span}_\CC \{ E_i\})$. Conversely, if   $A\in \WW_{\text{BR}}(\bar{\g}, f)$, without loss of generality, we can assume that $A\in F^p S(\mathcal{R}_-)^0 [\delta] \cap \text{ker } d_{[0]}.$ The corresponding element $[A] \in \gr^{p,-p} S(\mathcal{R}_-)[\delta]$ satisfies $d^{\gr} [A]=0$. Hence, by Lemma \ref{Lem:(R_-,d)good},  there is an element $B_1\in F^p S(\mathcal{R}_-)[\delta] \cap S(\CC[D]\otimes_{\CC} \text{Span}_\CC \{ E_i\})$ such that 
\[ A_1 = A+B_1 \in F^{p+\frac{1}{2}} S(\mathcal{R}_-)[\delta]\cap \text{ker } d_{[0]}.\]
Inductively, for $s \in \NN$, there is $B_s \in F^{p+\frac{s-1}{2}} S(\mathcal{R}_-)[\delta] \cap S(\CC[D]\otimes_{\CC} \text{Span}_\CC \{ E_i\})$ such that 
\[ A_s = A+B_1+\cdots B_s \in F^{p+\frac{s}{2}} S(\mathcal{R}_-)[\delta]\cap \text{ker } d_{[0]}.\]
Since $ S(\mathcal{R}_-)[\delta]$ is finite dimensional, $A_N=0$ for $N>\!>0$.  Hence  $A\in S(\CC[D]\otimes_{\CC} \text{Span}_\CC \{ E_i\}).$
\end{proof}

Now we can prove the main theorem in this section.

\begin{thm} \label{Thm:first_third_equiv}
We have
\[ \WW(\bar{\g},  f )\simeq \WW^{\ \mathrm{i}}_{\text{BR}}(\bar{\g},  f )\]
for the imaginary number $\mathrm{i}\in \CC$.
\end{thm}
\begin{proof}
Recall that $\mathcal{P}(\bar{\g})$ is  the  SUSY affine PVA associated with $\g$  and  $\mathcal{I}_{ f }$ is the differential algebra ideal of $\mathcal{P}(\bar{\g})$ generated by $\bar{n}-( f |n)$ for $n\in \n$. Then the image of $\{ \bar{a}_\chi  \bar{b}\}$ in $\mathcal{P}(\bar{\g})/\mathcal{I}_{ f }$ is 
\begin{equation} \label{Eqn:5.10_0409}
s(a)\big(\overline{\pi_{\leq 0}[a,b]} + ( f |[a,b]) + k \chi (a|b)\big).
 \end{equation}
If we denote
 \[ j_{\bar{a}}=\mathrm{i}^a \bar{a}\  \text{ and } \  j_{a}=\mathrm{i}^a a \text{ where } \mathrm{i}^a=\left\{ \begin{array}{ll} \mathrm{i} & \text{ if } a \text{ is odd, }\\
  1 & \text{ if } a \text{ is even, } \end{array}\right. \]
then, since $j_{\bar{\n}}=\{j_{\bar{n}}|n\in \n \}= \bar{\n}$, 
 \[  \WW(\bar{\g},  f )\simeq (\mathcal{P}(\bar{\g})/\mathcal{I}_{ f })^{\ad_\chi \,  j_{\bar{\n}}}.\]
Observe the following facts.
 \begin{enumerate}[(i)]
\item Equation (\ref{Eqn:5.10_0409}) can be rewritten as 
 \begin{equation*} \label{Eqn:5.11_0409}
 \begin{aligned}
 \, \{ j_{{\bar{a}}\, \chi \, } j_{\bar{b}}\} & = s(a,b)s(a) \left( \, j_{\overline{\pi_{\leq 0}[j_a,j_b]}}+ \mathrm{i} \,  ( f |[a,b]) \, \right)+k\chi(a|b)  
 \end{aligned}
 \end{equation*}
 \item For $u_\beta\in \n$, $j_{\bar{\beta}}=j_{\overline{u}_\beta}$ and $j_{ A }, j_{ B }\in \mathcal{P}(\bar{\g})$, we have 
 \[ \, \{ j_{\bar{\beta}\, \chi \, } j_{ A } j_{ B }\}= s( A ,\beta) j_{ A }\{ j_{\bar{\beta}\, \chi } j_{ B }\} +\{j_{\bar{\beta}\, \chi } j_{ A }\}j_{ B } \in \CC[\chi]\otimes \mathcal{P}(\bar{\g}). \]
 \item For $j_{ A }\in \mathcal{P}(\bar{\g})$, 
 \[  \{ j_{\bar{\beta}\, \chi \, }D j_{ A }\}=  s(\beta) (D+\chi) \{ j_{\bar{\beta}\, \chi \, } j_{ A }\}. \]
 \end{enumerate}

 On the other hand, associated with $\WW^{\ \mathrm{i}}_{\text{BR}}(\overline{\g},  f )$, we have the followings:
 \begin{enumerate}[(I)]
\item  For $J_{\bar{a}}$ for $a \in \bigoplus_{i\leq 0} \g(i)$, 
\begin{equation*} \label{Eqn:5.12_0409}
  d_{[0]}^{\ \mathrm{i}} ( J_{\bar{a}})= \sum_{\beta\in S} \left( \ s(a,\beta)s(\beta)\phi^{\bar{\beta}} J_{\overline{[u_\beta, a]}}+\mathrm{i}\, s(a,\beta)s(\beta) \phi^{\bar{\beta}} ( f |[u_\beta, a])-s(\beta) k D \phi^{\bar{\beta}} (u_\beta |a) \ \right).
  \end{equation*}
  \item For $J_{ A }$, $J_{ B }$ in $S(\CC[D]\otimes J_{\g_{\leq 0}})$, 
\[  d_{[0]}^{\ \mathrm{i}}(J_{ A } J_{ B })= s( A )J_{ A }\,  d_{[0]}^{\ \mathrm{i}}(J_{ B })+ d_{[0]}^{\ \mathrm{i}}(J_{ A }) J_{B}.\]
  \item For $J_{ A }$ in $S(\CC[D]\otimes J_{\g_{\leq 0}})$, 
  \[  d_{[0]}^{\ \mathrm{i}}(D J_{ A })= -D  d_{[0]}^{\ \mathrm{i}}(J_{ A }).\]
  \end{enumerate}
 
Let us consider the differential algebra isomorphism 
\[ \phi: \mathcal{P}(\g)/\mathcal{I}_{ f } \to S( \CC[D]\otimes J_{\g_{\leq 0}}), \quad j_{\overline{a}} \mapsto J_{\overline{a}}.\]
Suppose $a\in \g_{\leq 0}$. Denote the coefficient of $(-\chi^2)^{n_0} \chi^{n_1}$ for $n_0\in \ZZ_{\geq 0}$ and $n_1=0,1$ in $\{ j_{\bar{\beta}\, \chi} j_{\bar{a}}\}$  by $K^{n_0,n_1}_{(\beta, a)}$. Then, by (i) and (I), the coefficient of $\partial^{n_0} D^{n_1}\phi^\beta$ in $d_{[0]}^{\ \mathrm{i}}(J_{\bar{a}})$ is $(-s(\beta))^{n_1}\phi\big(K^{n_0,n_1}_{(\beta, a)}\big)$. Moreover, by comparing (ii) and (II) (resp.  (iii) and (III)), we have the following statement:  For any $j_{ A } \in S( \CC[D] \otimes  \g_{\leq 0})$, 
\begin{equation} \label{Eqn:5.11_0409}
\begin{aligned}
& \, \{ j_{\bar{\beta}\, \chi } j_{ A }\}= \sum_{n_0\in \ZZ_{\geq 0}, n_1=0,1}(-\chi^2)^{n_0} \chi^{n_1} K^{n_0, n_1}_{(\beta, A)}\  \in\   \CC[\chi]\otimes \mathcal{P}(\g)/\mathcal{I}_{ f } \\
&  \Longleftrightarrow  \quad d_{[0]} (\phi(j_{ A }))=  \sum_{n_0\in \ZZ_{\geq 0}, n_1=0,1}  (-s(\beta))^{n_1}\partial^{n_0} D^{n_1}\phi \left(K^{n_0, n_1}_{(\beta, A)}\right).
\end{aligned}
\end{equation}
Hence we have 
\[ \{j_{\bar{\beta}\, \chi}j_{ A }\}=0\in \CC[\chi]\otimes \mathcal{P}(\g)/\mathcal{I}_{ f }\ \text{ for all } \beta\in S \  \Longleftrightarrow\ d_{[0]}^{\ \mathrm{i}} \left( \phi(j_{ A }) \right)=0.\]
Recall that $\WW_\text{BR}^{\ \mathrm{i}}(\overline{\g},  f )=H^0(S(\mathcal{R}_-), d^{\ \mathrm{i}}_{[0]})$. Thus 
\[ \WW_I(\overline{\g},  f ) =  S( \CC[D]\otimes J_{\g_{\leq 0}}) \cap \ker (d_{[0]})\]
and $\WW_{\text{BR}}^{\ \mathrm{i}}(\overline{\g},  f ) \simeq \WW (\overline{\g},  f )$ as differential algebras.

Moreover, by comparing (\ref{Eqn:5.7_0409}) and (i), we conclude that $\WW(\bar{\g},  f )$ and $\WW_{\text{BR}}^{\ \ \mathrm{i}} (\overline{\g},  f )$ are isomorphic as SUSY PVAs.
\end{proof}

\section{Structures of SUSY classical affine W-algebras associated with Lie superalgebras}  \label{generators of SUSY W}

In this section, we use the notations from Section \ref{Sec:SUSY_W}. Recall that $\g$ is a simple finite Lie superalgebra with the subalgebra $\text{Span}_\CC\{E, e, H, f, F \}$ that is isomorphic to $\mathfrak{osp}(1|2)$.   Using the $\mathfrak{sl}_2$ representation theory, 
\[ \g^F= \g^f\oplus [e, \g^f] \quad \text{ for }\quad  \g^f=\ker \ad f\]
and, by \eqref{decomp_g}, we have
\[ \g= \g^f\oplus [e, \g].\] 

We can take  bases 
\[ \{r_j|j \in J^{f}\} \quad \text{ and } \quad  \{r^j|j \in J^{f}\}\]
of $\g^{f}$ and $\g^{e}$ such that $(r^i|r_j)=\delta_{i,j}$ and assume that the bases are homogeneous with respect to both parity and the $\ZZ/2$-grading on $\g$. 

Let us denote 
\begin{equation} \label{Eqn:r_m^i}
 r_m^i = (\ad f)^m r^i \text{ for } m\in \ZZ_{\geq 0}, i\in J^{f} 
\end{equation}
and 
\begin{equation}\label{Eqn:r_j^n}
 r_j^n= C_{j,n} (\ad e_{})^n r_j \text{ for } C_{j,n} \in \CC
 \end{equation}
such that $(r^i_m | r^n_j)= \delta_{m,n} \delta_{i,j}.$ 
\begin{lem}
In \eqref{Eqn:r_j^n}, we have 
\begin{equation} \label{Eqn:C_(j,n)}
\begin{aligned}
C_{j,n}= \left\{ \begin{array}{ll} \frac{1}{(m!)^2 {2\alpha_j \choose m}}  & \text{ if } n=2m \text{ is even }, \\ 
& \\
\frac{-s(j)}{(m+1)! m! {2\alpha_j \choose m+1}}  &  \text{ if } n=2m+1 \text{ is odd }.
\end{array} \right.
\end{aligned}
\end{equation}
\end{lem}
\begin{proof}
Let us denote $\tilde{r}_j^n:= (\text{ad} e)^n r_j$. Then $[f, \tilde{r}_j^n] = [-H, \tilde{r}_j^{n-1}]-[e, [f, \tilde{r}_j^{n-1}]].$
Hence for $A_n \in \CC$ such that $[f, \tilde{r}_j^n]= A_n\tilde{r}_j^{n-1}$, 
\begin{equation}
A_n = (2\alpha_j-n+1) - A_{n-1} \text{ for } n\geq 2
\end{equation}
and $A_1=2\alpha_j$. Thus, we obtain $A_{2m+1}= 2\alpha_j-m$ and $A_{2m}=-m$. Similarly, for $B_n \in \CC$ such that 
$[e, r_n^j]= B_n r_{n-1}^j$, we have $B_{2m+1}= -2\alpha_j+m$ and $B_{2m}=m$. 

Observe that 
\begin{equation}
\begin{aligned}
&  \left( r_n^j | \tilde{r}_j^n \right)  = ([f, r_{n-1}^j]|[e, \tilde{r}_j^{n-1}]) \\
 & = (-1)^{n-1} s(j) (-2\alpha_j +n-1) (r^j_{n-1}|\tilde{r}_j^{n-1})-([e, r_{n-1}^{j}]|[f, \tilde{r}_j^{n-1}]) \\
 & = (-1)^{n-1} s(j) (-2\alpha_j +n-1) (r^j_{n-1}|\tilde{r}_j^{n-1})-A_{n-1}B_{n-1} (r^j_{n-2}|\tilde{r}_j^{n-2}).
\end{aligned}
\end{equation}
We can check $(r^j_1 | \tilde{r}_j^1)= -s(j) 2\alpha_j$ by direct computations. Moreover, by induction, 
\[ \left( r_n^j | \tilde{r}_j^n \right) = \left\{ \begin{array}{ll}(m!)^2 {2\alpha_j \choose m}  & \text{ if } n=2m \text{ is even }, \\ 
-s(j)(m+1)! m! {2\alpha_j \choose m+1}  &  \text{ if } n=2m+1 \text{ is odd }.
\end{array} \right.\]
Since $\left( r_n^j | r_j^n \right)=1$, we get \eqref{Eqn:C_(j,n)}. 
\end{proof}

Consider the algebra homomorphisms
\begin{equation}
\begin{aligned}
&  \pi_S: \mathcal{P}(\bar{\g}_{\leq 0})\to  \mathcal{P}(\bar{\g}^{f}) \\
&  \rho_S: \mathcal{P}(\bar{\g}) \to \mathcal{P}(\bar{\g}_{\leq 0}), \quad \bar{a}\mapsto \bar{\pi}_{\leq 0}(\bar{a})+(f_{\text{od}}|a),
\end{aligned}
\end{equation}
where $\mathcal{P}(\mathcal{E})= S(\CC[D] \otimes \mathcal{E})$ and 
let 
\[ J^{f}_t := \{(i,m)\in J^{f}\times \ZZ_{\geq 0}| r_i^m \in \g_t \text{ and }  r^i_m \in \g_{-t}  \}.\]

\begin{lem} \label{Lem:SUSY generators}\hfill
\begin{enumerate}
\item If $(i,m) \in J^{f}_h$ and $(j,n)\in J^{f}_t$ then 
\begin{equation}
\rho_S\{\bar{r}^i_{m\ \chi} \bar{r}^n_j\}=
\left\{ 
\begin{array}{ll}
\, 0 & \text{ if } t-h>\frac{1}{2}, \\
\, (-1)^{m}s(i) \delta_{i,j}\delta_{n, m+1} &  \text{ if } t-h=\frac{1}{2}, \\
\, (-1)^{m}s(i) \big( \overline{[r_m^i, r_j^n]} + \delta_{i,j} \delta_{m,n} \, k \,  \chi \big) & \text{ if } t-h\leq 0,
\end{array}
\right.
\end{equation}
where $s(i)= s(r_i).$
\item 
If $\bar{r}\in \mathcal{P}(\bar{\g}^{f})  \big( \CC[D]\otimes \overline{[e, \g_{\leq -1/2}]} \, \big) $ satisfies 
\[ \pi_S \rho_S \{\bar{\n}_\chi \bar{r}\}=0 \quad \text{ for } \quad  n\in \g_{\geq 1/2}\]
then $\bar{r}=0$.
\end{enumerate}
\end{lem}

\begin{proof}
(1) can be obtained directly from the definition of $\rho_S$. To see (2), let us consider 
\[ \bar{r}= \sum_{(j,n)\in J^{f_{\text{od}}}_{\leq 0}} F_{j,n}(D) \bar{r}_j^n\]
where $F_{j,n}(D)=\sum_{t\in \ZZ_{\geq 0}} F_{j,n}^t D^t$ for $ F_{j,n}^t\in\mathcal{P}(\bar{\g}^{f}) $. Suppose $\bar{r}\neq 0$. Then there exists the largest number $k$ such that 
\[ (j,n) \in J_k^{f}, \quad F_{j,n}(D)\neq 0. \]
Then $r_{n-1}^j \in \g(-k+\frac{1}{2})\subset \g_{\geq 1/2}$ and 
\begin{equation*}
\begin{aligned}
& \pi_S \rho_S \{\bar{r}_{n-1}^j {}_\chi \bar{r} \}=  \sum_{t\in \ZZ_{\geq 0}} s(F_{j,n}^t, r_{n-1}^j) F_{j,n}^t \{ \bar{r}_{n-1}^j {}_\chi D^t \bar{r}_j^t \}   \\
& =  \sum_{t\in \ZZ_{\geq 0}}  s(F_{j,n}^t, r_{n-1}^j) s(r_{n-1}^j)^{t+1}F_{j,n}^t (\chi+D)^t \neq 0.
\end{aligned}
\end{equation*}
Hence, we have proved the lemma.
\end{proof}

By Lemma \ref{Lem:SUSY generators} and the structure theory of SUSY W-algebras (Proposition \ref{Prop:5.5_0408}), we have the unique differential algebra isomorphism 
\[ \omega_S : \mathcal{P} (\bar{\g}^{f}) \to \mathcal{W}(\bar{\g},f), \qquad \bar{a}\mapsto \omega_S({\bar{a}}) \]
where $\omega_S({\bar{a}}) = \bar{a}+\gamma_S(\bar{a})+\gamma_S^{\geq 2}(\bar{a})$ for \[ \gamma_S(\bar{a})\in \mathcal{P}(\bar{\g}^{f}) \otimes  \big( \CC[D]\overline{[e_{\text{od}}, \g_{\leq -1/2}]} \,  \big) , \quad \gamma_S^{n}(a)\in \mathcal{P}(\bar{\g}^{f})\otimes  (\CC[D] \overline{[e_{\text{od}}, \g_{\leq -1/2}]})^{\otimes n}\]
 and $\gamma_S^{\geq 2}(a)= \sum_{i\geq 2} \gamma_S^i(a)$.

\begin{thm} \label{Thm:SUSY_generator}
Let $\pi_{\g^f}: \g \to \g^f$ be the projection map and  
us denote $\bar{g}^{\sharp_S} := \overline{\pi_{\g^f}(g)}$ for $g\in \g$. For $a\in \g^{f}\cap \g(-\alpha)$, we have  
\begin{equation}
\begin{aligned}
&  \gamma_S(\bar{a})=  \sum_{p\in \ZZ_{\geq 0}} \sum_{\substack{ (-\alpha-\frac{1}{2}) \prec_S ( j_0, n_0) \prec_S \cdots \\ \ \cdots  \prec_S ( j_p, n_p) \prec_S 0 } }  \left(  \overline{[a, r_{n_0}^{j_0}]}^{\sharp_S}- ( a| r_{n_0}^{j_0} ) kD \right)\\
 &  \hskip 3cm \left[ \prod_{ t=1, \cdots, p } \left( \overline{[r_{j_{t-1}}^{n_{t-1}+1}, r_{n_t}^{j_t}] }^{\sharp_S}- (r_{j_{t-1}}^{n_{t-1}+1}| r_{n_t}^{j_t})kD \right) \right] \quad   \bar{r}_{j_p}^{n_p+1},
 \end{aligned}
\end{equation}
where
\[ ( j_t, n_t)\prec_S (j_{t+1}, n_{t+1}) \ \text{ if and only if } \ 
\left\{ \begin{array}{l} k_{t+1} -k_t \geq \frac{1}{2} \text{ where } \\
(j_t, n_t)\in J^{f}_{k_t} \text{ and } (j_{t+1}, n_{t+1})\in J^{f}_{k_{t+1}} \end{array}\right. \]
and 
\[  \textstyle  -\alpha-\frac{1}{2} \prec_S (j_0,n_0) \quad   \text{ if and only if } \quad  (j_0, n_0) \in J_{k_0}^{f} \text{ for } k_0 \geq -\alpha.\]
\end{thm}

\begin{proof}
By Lemma \ref{Lem:SUSY generators} (2) and 
\[ \pi_S \rho_S \{\bar{r}_{m}^i{}_\chi \gamma_S^{\geq 2}(\bar{a})\}=0\]
it is enough to show that  $\pi_S \rho_S \{\bar{r}_{m}^i{}_\chi \bar{a}+{\gamma_S(\bar{a})} \}=0.$ 
Let us write

\begin{equation} \label{Eqn:gamma_S[p]}
\begin{aligned}
& \gamma_S(\bar{a})[-1]: =\bar{a} \\
&   \gamma_S(\bar{a})[0] : =  \sum_{ (-\alpha-\frac{1}{2}) \prec_S ( j_0, n_0) } \left(  \overline{[a, r_{n_0}^{j_0} ]}^{\sharp_S}- ( a| r_{n_0}^{j_0} ) kD \right)\bar{r}_{j_0}^{n_0+1} \\
 & \gamma_S(\bar{a})[p]:= \sum_{\substack{ (-\alpha-\frac{1}{2}) \prec_S ( j_0, n_0) \prec_S \cdots \\ \ \cdots  \prec_S ( j_p, n_p) \prec_S 0 } }  \left(  \overline{[a, r_{n_0}^{j_0} ]}^{\sharp_S}- ( a| r_{n_0}^{j_0} ) kD \right)\\
 &  \hskip 3cm \left[ \prod_{ t=1, \cdots, p } \left( \overline{[r_{j_{t-1}}^{n_{t-1}+1}, r_{n_t}^{j_t}] }^{\sharp_S}- (r_{j_{t-1}}^{n_{t-1}+1}| r_{n_t}^{j_t})kD \right) \right] \quad   \bar{r}_{j_p}^{n_p+1}
 \end{aligned}
\end{equation}
for $p\geq 1$. Then 
\begin{equation*}
\begin{aligned}
& \pi_S \rho_S \{\bar{r}_m^i{}_\chi \gamma_S(\bar{a})[-1]\}= -s(r_m^i, a) s(r_m^i)\big( \overline{[a,r_m^i]}^{\sharp_S}-k\chi(a|r_m^i)\big), \label{[0]} \\
& \pi_S \rho_S \{\bar{r}_m^i{}_\chi \gamma_S(\bar{a})[0]\}\\
& \qquad =\sum_{ -\alpha -\frac{1}{2} \prec_S ( j_0, n_0)}s(r_m^i, a r_{n_0}^{j_0})s(r_m^i) \left(\overline{[a, r_{n_0}^{j_0}]}^{\sharp_S}- (a| r_{n_0}^{j_0})k(\chi+D) \right)  \pi_S \rho_S  \{\bar{r}_m^i{}_\chi  \bar{r}_{j_0}^{n_0+1}\} \\
& \qquad  = \sum_{-\alpha -\frac{1}{2} \prec_S ( j_0, n_0)}-s(r_m^i, a) s(r_m^i)\left(\overline{[a, r_{n_0}^{j_0}]}^{\sharp_S}- (a| r_{n_0}^{j_0})k(\chi+D) \right)  \left(\overline{[r_{j_0}^{n_0+1}, r_{m}^{i}]}^{\sharp_S}- k\chi(r_{j_0}^{n_0+1}| r_{m}^{i}) \right) \\
& \hskip 5cm + s(r_m^i, a) s(r_m^i)\big( \overline{[a,r_m^i]}^{\sharp_S}-k\chi(a|r_m^i)\big). 
\end{aligned}
\end{equation*}
Hence,
\begin{equation*}
\begin{aligned}
& \quad  \pi_S \rho_S \{\bar{r}_m^i{}_\chi \gamma_S(\bar{a})[0]+ \gamma_S(\bar{a})[1]\}\\
& = \sum_{ -\alpha-\frac{1}{2} \prec_S ( j_0, n_0)}-s(r_m^i, a) s(r_m^i) \left(\overline{[a, r_{n_0}^{j_0}]}^{\sharp_S}- (a| r_{n_0}^{j_0})k(\chi+D) \right)  \left(\overline{[r_{j_0}^{n_0+1}, r_{m}^{i}]}^{\sharp_S}- (r_{j_0}^{n_0+1}| r_{m}^{i})k\chi \right).
 \end{aligned}
 \end{equation*}
 Inductively, one can show that  \[\pi_S \rho_S \{\bar{r}_m^i{}_\chi \gamma_S(\bar{a})[p]\}  = -s(r_m^i, a) s(r_m^i)  \Gamma [p]+s(r_m^i, a) s(r_m^i) \Gamma [p-1],\] where 
 \begin{equation*}
 \begin{aligned}
 \Gamma[p]   &  = \sum_{\substack{ -\alpha-\frac{1}{2} \prec_S ( j_0, n_0) \prec_S \cdots \\ \ \cdots \prec_S (j_{p-1}, n_{p-1}) \prec_S ( j_p, n_p) } }  \left(  \overline{[a, r_{n_0}^{j_0} ]}^{\sharp_S}- ( a| r_{n_0}^{j_0} ) k(\chi+D) \right)\\
 &  \left[  \prod_{ t=1, \cdots, p }  \left( \overline{[r_{j_{t-1}}^{n_{t-1}+1}, r_{n_t}^{j_t}] }^{\sharp_S}- (r_{j_{t-1}}^{n_{t-1}+1}| r_{n_t}^{j_t})k(\chi+D) \right)  \right] \left( \overline{[r_{j_{p}}^{n_{p}+1}, r_{m}^{i}] }^{\sharp_S}- (r_{j_{p}}^{n_{p}+1}| r_{m}^{i})k\chi\right). 
 \end{aligned}
\end{equation*}
Hence, 
\[ \pi_S \rho_S \{\bar{r}_{m}^i{}_\chi  \gamma_S(\bar{a})\}=\sum_{p\in \ZZ_{\geq -1}} \pi_S \rho_S \{\bar{r}_m^i{}_\chi \gamma_S(\bar{a})[p]\}=0.\]
\end{proof}

\begin{lem} \label{SUSY_relation_lemma}
For $t\in \frac{\ZZ}{2}$, we have
\[ \sum_{(i,m)\in J_{-t}^{f}}r_m^i \otimes r_i^{m+1}= -\sum_{(j,n)\in J_{t-\frac{1}{2}}^{f}} r_{j}^{n+1} \otimes r_n^j. \]
\end{lem}
\begin{proof}
It suffices to show that for $[f, a] \in \g_{-t}$ and $[f, b] \in \g_{t-1/2}$, 
\begin{equation} \label{Eqn:SUSY_relation_lemma}
\sum_{(i,m)\in J_{-t}^{f}}([f, a]|r_m^i) ([f, b]|r_i^{m+1})= -\sum_{(j,n)\in J_{t-\frac{1}{2}}^{f}} ([f, a]| r_{j}^{n+1}) ([f, b]| r_n^j).
\end{equation}
Since 
\begin{equation*}
 \sum_{(i,m)\in J_{-t}^{f}}([f, a]|r_m^i) ([f, b]|r_i^{m+1})=   \sum_{(i,m)\in J_{-t}^{f}}-(r_{m+1}^{i} |a) ([f, b]|r_i^{m+1}) =-([f, b]|a)
\end{equation*}
and 
\begin{equation*}
 -\sum_{(j,n)\in J_{t-\frac{1}{2}}^{f}} ([f, a]| r_{j}^{n+1}) ([f, b]| r_n^j)= \sum_{(j,n)\in J_{t-\frac{1}{2}}^{f}} ([f, a]| r_{j}^{n+1})(r_{n+1}^j |b)=([f,a]|b),
\end{equation*}
we have \eqref{Eqn:SUSY_relation_lemma}.
\end{proof}

\begin{thm} \label{Thm:Pisson chi bracket}
Let $a\in \g^f_{-t_1}$ and $b\in \g^f_{-t_2}.$ Then 
\begin{equation}
\begin{aligned}
& \{\omega_S (\bar{a}){}_\chi \omega_S(\bar{b}) \} = s(a)( \overline{[a,b]}+\chi k(a|b))\\
 &\hskip 1cm-s(a,b)s(a) \sum_{p\in\ZZ_{\geq 0}} \sum_{\substack{ -t_2-\frac{1}{2} \prec_S ( j_0, n_0) \prec_S \cdots \\ \ \cdots \prec_S (j_{p}, n_{p}) \prec_S  t_1} }  \left(\omega_S( \overline{[ b,r_{m_0}^{i_0} ]}^{\sharp_S})-k(b| r_{m_0}^{i_0}) (D+\chi)\right) \\
& \hskip 3cm\left[  \prod_{t=1,2, \cdots, p} \left(\omega_S( \overline{[ r_{i_{t-1}}^{m_{t-1}+1},r_{m_t}^{i_t} ]}^{\sharp_S})-k( r_{i_{t-1}}^{m_{t-1}+1}| r_{m_t}^{i_t}) (D+\chi)\right)\right]  \\
& \hskip 6cm \left(\omega_S( \overline{[ r_{i_p}^{m_p+1},a ]}^{\sharp_S})-k( r_{i_p}^{m_p+1}| a) (D+\chi)\right).
\end{aligned}
\end{equation}
\end{thm}

\begin{proof}
Recall the notations in \eqref{Eqn:gamma_S[p]} and, for simplicity, we denote by $\overrightarrow{(i, m)_0^p}$ the sequence $(i_0, m_0), \cdots, (i_p,m_p)$ such that 
\[ (i_0, m_0) \prec_S ( i_1, m_1) \prec_S \cdots   \prec_S (i_{p-1}, m_{p-1}) \prec_S ( i_p, m_p)\]
and define $\overrightarrow{(j, n)_0^{q}}$ similarly.

It suffices to show that 
\begin{equation} \label{Eqn:SUSY_relations_5}
\begin{aligned}
&\pi_S \rho_S \{ \gamma_S(\bar{a})[p]{}_\chi  \gamma_S(\bar{b})[q]\}= -s(a,b)s(a)   \sum_{\overrightarrow{(i, m)_0^{p+q+1}} } \left( \overline{[ b,r_{m_0}^{i_0} ]}^{\sharp_S}-k(b| r_{m_0}^{i_0}) (D+\chi)\right)  \\
& \hskip 3cm\left[ \prod_{t=1,2, \cdots, p+q+1} \left( \overline{[ r_{i_{t-1}}^{m_{t-1}+1},r_{m_t}^{i_t} ]}^{\sharp_S}-k( r_{i_{t-1}}^{m_{t-1}+1}| r_{m_s}^{i_s}) (D+\chi)\right)\right]\\
 &  \hskip 6cm\left( \overline{[ r_{i_{p+q+1}}^{m_{p+q+1}+1},a ]}^{\sharp_S}-k( r_{i_{p+q+1}}^{m_{p+q+1}+1}| a) (D+\chi)\right).
\end{aligned}
\end{equation}
By the master formula, we have
\begin{equation} \label{Eqn:SUSY_relations_6}
\begin{aligned}
& \pi_S \rho_S \{ \gamma_S(\bar{a})[p]{}_\chi  \gamma_S(\bar{b})[q]\} =\sum_{\substack{-t_1-\frac{1}{2}\prec_S\overrightarrow{(i,m)_0^p} \prec_S 0 \\ -t_2-\frac{1}{2}\prec_S\overrightarrow{(j,n)_0^q} \prec_S 0 }} s(a,\bar{b}) s(a, r_{j_q}^{n_q}) \left( \overline{[b, r_{n_0}^{j_o}] }^{\sharp_S}- (b| r_{n_0}^{j_0})(\chi+D) \right)\\
  &  \hskip 1cm \left[ \prod_{ t=1, \cdots, q } \left( \overline{[r_{j_{t-1}}^{n_{t-1}+1}, r_{n_t}^{j_t}] }^{\sharp_S}- (r_{j_{t-1}}^{n_{t-1}+1}| r_{n_t}^{j_t})(\chi+D) \right) \right] \pi_S \rho_S \bigg\{ \gamma_S(\bar{a})[p]_{\chi}\,  \bar{r}_{j_q}^{n_q+1}\bigg\}.
\end{aligned}
\end{equation}
The $\chi$-bracket in the last line of \eqref{Eqn:SUSY_relations_6} can be written as 
\begin{equation} \label{Eqn:SUSY_relations_7}
\begin{aligned}
& \sum_{-t_1-\frac{1}{2}\prec_S\overrightarrow{(i,m)_0^p} \prec_S 0 } \pi_S \rho_S\bigg\{ \gamma_S(\bar{a})[p]_{\chi }\,  \bar{r}_{j_q}^{n_q+1}\bigg\}\\
& = s(a, r_{i_0}^{m_0}) \left[\prod_{t=1, \cdots, p} s(r_{i_{t-1}}^{m_{t-1}+1}, r_{i_t}^{m_t})\right] s(r_{j_q}^{n_q}, \bar{a} r_{i_p}^{m_p}) 
 \pi_S \rho_S\big\{\bar{r}_{i_p}^{m_p+1}{}_{\chi+D}\,  \bar{r}_{j_q}^{n_q+1}\big\}_{\to}\\
& \left[ \prod_{ t=1, \cdots, p }^{\leftarrow}   \bigg(\overline{[r_{i_{t-1}}^{m_{t-1}+1}, r_{m_t}^{i_t}] }^{\sharp_S}+(r_{i_{t-1}}^{m_{t-1}+1}| r_{m_t}^{i_t})k(D+\chi) \bigg) \right]  \big(\overline{[a, r_{m_0}^{i_0}] }^{\sharp_S}+(a| r_{m_0}^{i_0})k(\chi+D) \big),\\
& =(-1)^{p+2}s(r_{j_q}^{n_q}, a)  \bigg(\overline{[ r_{j_q}^{n_q+1},{r}_{i_p}^{m_p+1} ]}-k(\chi+D) ( r_{j_q}^{n_q+1}| {r}_{i_p}^{m_p+1})\bigg)\\
& \left[ \prod_{ t=1, \cdots, p }^{\leftarrow}   \bigg(\overline{[ r_{m_t}^{i_t}, r_{i_{t-1}}^{m_{t-1}+1}] }^{\sharp_S}-( r_{m_t}^{i_t})|r_{i_{t-1}}^{m_{t-1}+1})k(D+\chi) \bigg) \right]  \big(\overline{[r_{m_0}^{i_0}, a ] }^{\sharp_S}-( r_{m_0}^{i_0}|a)k(\chi+D) \big)
\end{aligned}
\end{equation}
where  $\displaystyle \prod_{ t=1, \cdots, p }^{\leftarrow}A_s:= A_p A_{p-1}\cdots A_2 A_1$.  
By \eqref{Eqn:SUSY_relations_6}-\eqref{Eqn:SUSY_relations_7},
\begin{equation}
\begin{aligned}
& \pi_S \rho_S \{ \gamma_S(\bar{a})[p]{}_\chi \gamma_S(\bar{b})[q]\}=s(a,b)s(a)(-1)^{p} \sum_{-t_1-\frac{1}{2}\prec_S\overrightarrow{(i,m)_0^p} \prec_S 0} \ \ \sum_{ -t_2-\frac{1}{2}\prec_S\overrightarrow{(j,n)_0^q} \prec_S 0 }
\\ &  \left( \overline{[b, r_{n_0}^{j_0}] }^{\sharp_S}- (b| r_{n_0}^{j_0})(\chi+D) \right) \left[\prod_{ t=1, \cdots, q } \left( \overline{[r_{j_{t-1}}^{n_{t-1}+1}, r_{n_t}^{j_t}] }^{\sharp_S}- (r_{j_{t-1}}^{n_{t-1}+1}| r_{n_t}^{j_t})(\chi+D) \right)\right]\\
 & \qquad \bigg(\overline{[ r_{j_q}^{n_q+1},{r}_{i_p}^{m_p+1} ]}-(\chi+D) ( r_{j_q}^{n_q+1}| {r}_{i_p}^{m_p+1})\bigg)\\
 & \left[\prod_{ t'=1, \cdots, p }^{\leftarrow}\bigg(\overline{[ r_{m_{t'}}^{i_{t'}}, r_{i_{t'-1}}^{m_{t'-1}+1}] }^{\sharp_S}
-(r_{m_{t'}}^{i_{t'}}|r_{i_{t'-1}}^{m_{t'-1}+1})(D+\chi)\bigg) \right] \bigg(\overline{[ r_{m_0}^{i_0}, a] }^{\sharp_S}
-(r_{m_0}^{i_0}|a)(\chi)\bigg).
\end{aligned}
\end{equation}
Now, by Lemma \ref{SUSY_relation_lemma}, we get \eqref{Eqn:SUSY_relations_5}. Thus, the theorem follows.
\end{proof}

\end{document}